\documentclass{amsart}
\usepackage{amsmath,amssymb,amscd,amsthm,amsxtra,array}
\usepackage[
colorlinks=true,
linkcolor=blue,
 citecolor=blue,
  urlcolor=blue,
]{hyperref}

\usepackage{tikz}
\usepackage{amssymb}

 \newtheorem{thm}{Theorem}[section]
 \newtheorem{cor}[thm]{Corollary}
 \newtheorem{lem}[thm]{Lemma}
 \newtheorem{prop}[thm]{Proposition}
 \newtheorem{cla}[thm]{Claim}
 \theoremstyle{definition}
 \newtheorem{defn}[thm]{Definition}
 \theoremstyle{remark}
 \newtheorem{rem}[thm]{Remark}
 \newtheorem{conj}[thm]{Conjecture}
  
 \newtheorem{ex}[thm]{Example}
 \numberwithin{equation}{section}

\newcommand{\RR}{\mathbb{R}^{2}}
\newcommand{\h}{\mathfrak{h}}
\newcommand{\g}{\mathfrak{g}}
\newcommand{\e}{\varepsilon}

\newcommand{\N}{\mathbb{N}}
\newcommand{\D}{\displaystyle}

\newcommand{\R}{\mathbb{R}}

\newcommand{\LL}{\mathbb{L}}
\newcommand{\RN}{{\mathbb{R}^{n}}}
\newcommand{\s}{\mathbb{S}}
\newcommand{\sub}{\subseteq}
\newcommand{\diam}{\text{diam}}

\newcommand{\Q}{\mathbb{Q}}

\renewcommand{\H}{\mathbb{H}}

\newcommand{\NN}{\mathcal{N}_\delta}

\renewcommand{\a}{\mathfrak{a}}
\renewcommand{\b}{\mathfrak{b}}
\newcommand{\f}{\mathfrak{f}}
\newcommand{\n}{\mathfrak{n}}

\newcommand{\id}{\textbf{id}}
\newcommand{\HD}{\mathbb{H}}
\newcommand{\lbd}{\underline{\dim}_B}

\newcommand{\hh}[1]{\mathcal{H}_{\delta}^{#1}}
\newcommand{\HH}[1]{\mathcal{H}^{#1}}

\renewcommand{\id}{\textbf{id}}

\begin{document}
\title[Furstenberg sets and Hausdorff measures]
 {Refined size estimates for Furstenberg sets via Hausdorff measures: a survey of some recent results}
\author{Ezequiel Rela}

\address{Departamento de An\'alisis Matem\'atico, Facultad de Matem\'aticas,
Universidad de Sevilla, 41080 Sevilla, Spain.}
\email{erela@us.es}

\thanks{This work was completed with the support of the Departmento de An\'alisis Matem\'atico at the Facultad de Matem\'aticas, Universidad de Sevilla, Spain.}


\date{\today}

\begin{abstract}
In this survey  we collect and discuss some recent results on the so called ``Furstenberg set problem'', which in its classical form concerns the estimates of the Hausdorff dimension ($\dim_H$) of the sets in the $F_\alpha$-class: for a given $\alpha\in(0,1]$, a set $E\sub\RR$ is in the $F_\alpha$-class if for each $e\in\s$ there exists a unit line segment $\ell_e$ in the direction of $e$ such that $\dim_H(\ell\cap E)\ge\alpha$. For $\alpha=1$, this problem is essentially equivalent to the ``Kakeya needle problem''. Define 
$\gamma(\alpha)=\inf \left\{\dim_H(E): E\in F_\alpha\right\}$. The best known results on $\gamma(\alpha)$ are the following inequalities:
\begin{equation*}
\max\left\{1/2+\alpha ; 2\alpha\right\}\le \gamma(\alpha)\le(1+3\alpha)/2.  
\end{equation*} 
In this work we approach this problem from a more general point of view, in terms of a generalized Hausdorff measure $\HH{h}$ associated with the dimension function $h$. We define the class $F_h$ of Furstenberg sets associated to a given dimension function $h$. The natural requirement for a set $E$ to belong to $F_h$, is that $\HH{h}(\ell_e\cap E)>0$ for each direction. We generalize the known results in terms of ``logarithmic gaps'' and obtain analogues to the estimates given above. Moreover, these analogues allow us to extend our results to the endpoint $\alpha=0$. For the upper bounds we exhibit an explicit construction of $F_h$-sets which are small enough.
To that end we adapt and prove some results on Diophantine Approximation about the the dimension of a set of ``well approximable numbers''.

We also obtain results about the dimension of Furstenberg sets in the class $F_{\alpha\beta}$, defined analogously to the class $F_\alpha$ but only for a fractal set $L\subset \s$ of directions such that $\dim_H(L)\ge\beta$. We prove analogous inequalities reflecting the interplay between $\alpha$ and $\beta$. This problem is also studied in the general scenario of Hausdorff measures.

\end{abstract}

\maketitle

\section{Introduction}
In many situations in geometric measure theory, one wants to determine the size of a given set or a given class of sets identified by some geometric properties. Throughout this survey, size will mean Hausdorff dimension, denoted by $\dim_H$. The main purpose of the present expository work is to exhibit some recent results on the study of dimension estimates for \emph{Furstenberg sets}, most of them contained in \cite{mr10},\cite{mr12} and \cite{mr13}. Some related topics and the history of this problem is also presented. As far as we know, there is no other work in the literature collecting the known results about this problem. We begin with the definition of the \emph{Furstenberg classes}. 
\begin{defn}
For $\alpha$ in $(0,1]$, a subset $E$ of $\RR$ is called \textit{Furstenberg set} or $F_\alpha$-set if for each direction $e$ in the unit circle there is a line segment $\ell_e$ in the direction of $e$ such that the Hausdorff dimension of the set $E\cap\ell_e$ is equal or greater than $\alpha$. 
\end{defn}
We will also say that such set $E$ belongs to the class $F_\alpha$. It is known that for any $F_\alpha$-set $E\subseteq \RR$ the Hausdorff dimension must satisfy the inequality $\dim_H(E)\ge\max\{2\alpha,\alpha+\frac{1}{2}\}$. On the other hand, there are examples of $F_\alpha$-sets $E$ such that  $\dim_H(E)\le \frac{1}{2}+\frac{3}{2}\alpha$. 
If we denote by
\begin{equation*}
  \gamma(\alpha)=\inf \{\dim_H(E): E\in F_\alpha\}, 
\end{equation*}
then the Furstenberg problem is to determine $\gamma(\alpha)$. The best known bounds on $\gamma(\alpha)$ so far are
\begin{equation}\label{eq:dim}
\max\{  2\alpha ;\frac{1}{2}+\alpha\}\le \gamma(\alpha)\le\frac{1}{2}+\frac{3}{2}\alpha,\qquad \alpha\in(0,1].
\end{equation}

\subsection{History and related problems}

The Furstenberg problem appears for the first time in the work of Harry Furstenberg in \cite{fur70}, regarding the problem of estimating the size of the intersection of fractal sets. Main references on this matter are \cite{wol99b},  \cite{wol99a} and \cite{wol02}. See also \cite{kt01} for a discretized version of this problem. In this last article, the authors study some connections between the Furstenberg problem and two other very famous problems: the Falconer distance problem and the Erd\"os ring problem.

Originally, in \cite{fur70} Furstenberg  dealt with the problem of transversality of sets. Briefly, two closed subsets $A,B\subset\R$ are called \emph{transverse} if 
\begin{equation*}
 \dim_H(A\cap B)\le\max\{\dim_H A +\dim_H B -1, 0\}. 
\end{equation*} 
In addition, they will be called \emph{strongly transverse} if every translate $A+t$ of $A$ is transverse to $B$. More generally, the problem of the transversality between the dilations $uA$ of $A$ and $B$ was considered. In this case the relevant quantity is $\dim_H(uA+t\cap B)$. This is where the connection pops in, since the dimension of this intersection can be seen as the dimension of the set $(A\times B)\cap \ell_{ut}$, where the line $\ell_{ut}$ in $\RR$ is defined by the equation $y=ux+t$. In addition, Furstenberg proves, with some \emph{invariance} hypothesis on $A$ and $B$, the following: if the product $A\times B$ intersect one (and it suffices with only one) line in some direction on a set of dimension at least $\alpha$, then for almost all directions the set $A\times B$ intersects a line in that direction also in a set of dimension at least $\alpha$. Therefore, in that case the product is an $F_\alpha$-set. Hence, any non trivial lower bound on the class $F_\alpha$ implies a lower bound for the dimension of the product $A\times B$ in this particular case.

We now make the connection between the Furstenberg problem and the Falconer and Erd\"os problems more precise. We begin with the formulation of the Falconer distance problem. For a compact set $K\sub\RR$, define the \emph{distance set} dist($K$) by
\begin{equation*}
 \text{dist}(K):=\{|x-y|: x,y\in K\}.
\end{equation*} 
The conjecture here is that $\dim_H(\text{dist}(K))=1$ whenever $\dim_H(K)\ge 1$. In the direction of proving this conjecture, it was shown by Bourgain in \cite{bou94} that the conclusion holds for any $K$ of $\dim_H(K)\ge\frac{13}{9}$, improved later by Wolff in \cite{wol99a} to $\dim_H(K)\ge\frac{4}{3}$. On the other hand, Mattila shows in \cite{mat87} that if we assume that $\dim_H(K)\ge 1$, then $\dim_H(\text{dist}(K))\ge\frac{1}{2}$. One may ask if there is an absolute constant $c>0$ such that $\dim_H(\text{dist}(K))\ge\frac{1}{2}+c_0$ whenever $K$ is compact and satisfies $\dim_H(K)\ge1$. The Erd\"os ring problem, roughly speaking, asks about the existence of a Borel subring $R$ of $\R$ such that $0<\dim_H(R)<1$.

The connection has been established only for some discretized version of the above three problems (see \cite{kt01}).  Consider the special case of Furstenberg sets belonging to the $F_\frac{1}{2}$ class. Note that for this family the two lower bounds for the Hausdorff dimension of Furstenberg sets coincide to become $\gamma(\frac{1}{2})\ge 1$. Essentially, the existence of the constant $c_0$ in the Falconer distance problem mentioned above is equivalent to the existence of another constant $c_1$ such that any $F_\frac{1}{2}$-set $E$ must have $\dim_H(E)\ge 1 +c_1$. In addition, any of this two conditions would imply the non existence of a Borel subring of $R$ of Hausdorff dimension exactly $\frac{1}{2}$. 

In addition, Wolff state without proof in \cite{wol99a} that there is a relation between the Furstenberg problem and the rate of decay of circular means of Fourier transforms of measures. Later, in \cite{wol02}, appears the proof of the following fact. Let $\mu$ be a measure and define the $s$-dimensional energy $I_s(\mu)$ as 
$$
I_s(\mu):=\int \frac{d\mu(x)d\mu(y)}{|x-y|^s}.
$$
Define also $\sigma_1(s)$ to be the supremum of all the numbers $\sigma$ such that exists a constant $C$ with
$$ 
\int_{-\pi}^\pi \left| \hat{\mu}(R e^{i\theta}) d\theta\right|\le C R^{-\sigma}\sqrt{I_s(\mu)},
$$
for all positive measures with finite $s$-energy supported in the unit disc and all $R\ge 1$.			
Then the following relation holds: 
$$1-\alpha-\gamma(\alpha)-4\sigma_1(\gamma(\alpha))\ge 1.$$

For the particular case of $\alpha=1$, when we require the set to contain a whole line segment in each direction, we actually are in the presence of the much more famous Kakeya problem. A Kakeya set (or Besicovitch set) is a compact set $E\sub\RN$ that contains a unit line segment in every possible direction. The question here is about the minimal size for the class of Kakeya sets. Besicovitch \cite{bes19} proved that for all $n\ge2$, there exist Besicovitch sets of Lebesgue measure zero in $\R^n$.  		

Originally, Kakeya \cite{fk17} asks which is the possible minimal area that permits to \emph{continuously} turn around  a unit line segment in the plane and in \cite{bes28} Besicovitch actually  shows that the continuous movement can be achieved using an arbitrary small area by the method known as \emph{shif\-ting triangles} or \emph{Perron's trees}. 

The next question, which is relevant for our work, is the {\em unsolved} ``Kakeya conjecture'' which asserts that these sets, although they can be small with respect to the Lebesgue measure, must have full Hausdorff dimension. 
The conjecture was proven by Davies \cite{dav71} in $\RR$: all Kakeya sets in $\RR$ have dimension 2. In higher dimensions the Kakeya problem is still open, and one of the best known bounds appears in \cite{wol99b} and states that any Kakeya set $E\sub \R^n$ must satisfy the bound $\dim_H(E)\ge \frac{n+2}{2}$. 

These kind of geometric-combinatorial problems have deep implications in many different areas of general mathematics. Some of the connections to other subjects include Bochner-Riesz multipliers, restrictions estimates for the Fourier transform and also partial differential equations. For example, it has been shown that a positive answer to the Restriction Conjecture  for the sphere $\s^{n-1}$ would imply that any Kakeya set in $\RN$ must have full dimension, and therefore solve the Kakeya conjecture (see for example \cite{wol99b}) .

\subsection{Our approach}

In this work we study the Furstenberg problem using generalized Hausdorff measures. This approach is motivated by the well known fact that knowing the value of the dimension of a given set is not telling us yet anything about the corresponding measure at this critical dimension. In fact, if $\HH{s}$ is the Hausdorff $s$ measure of an $s$-dimensional set $E$, $\HH{s}(E)$ can be $0$, $\infty$ or finite. The case of a set $E$ with $0<\HH{s}(E)<+\infty$ is of special interest. We refer to it as an $s$-set, considering it as \emph{truly} $s$-dimensional. For, if a set $E$ with $\dim_H(E)=s$ has non $\sigma$-finite $\HH{s}$-measure, it is still too big to be correctly measured by $\HH{s}$. Analogously, the case of null measure reflects that the set is too thin to be measured by $\HH{s}$. To solve (partially) this problem, the appropriate tools are the ``generalized Hausdorff Measures'' introduced by Felix Hausdorff in his seminal paper \cite{hau18} in 1918.  For any {\em dimension function}, i.e. a function belonging to the set
\begin{equation*}
 \mathbb{H}:=\{h:[0,\infty)\to[0:\infty), \text{non-decreasing, continuous, } h(0)=0 \},
\end{equation*} 
he defines
\begin{equation*}
	\hh{h}(E)=\inf\left\{\sum_i h(\diam(E_i)):E\subset\bigcup_i^\infty E_i,  \diam(E_i)<\delta \right\}
\end{equation*}
and 
\begin{equation*}
	\HH{h}(E)=\sup_{\delta>0 }\hh{h}(E).
\end{equation*} 
Note  that if $h_\alpha(x):=x^\alpha$, we actually recover the previous measure since $\HH{h_\alpha}=\HH{\alpha}$. 
We now have a finer criteria to classify sets by a notion of size. If one only looks at the power functions, there is a natural total order given by the exponents. In $\mathbb{H}$ we also have a natural notion of order, but we can only obtain a \textit{partial} order.
\begin{defn}\label{def:order}
Let $g,h$ be two dimension functions. We will say that $g$ is dimensionally smaller than $h$ and write $g\prec h $ if and only if
	\begin{equation*}
	\lim_{x\to 0^+}\dfrac{h(x)}{g(x)}=0.
	\end{equation*}	
\end{defn}
We note that the speed of convergence to zero can be seen as a notion of distance between $g$ and $h$. The important subclass of those $h\in\mathbb{H}$ that satisfy a doubling condition will be denoted by $\mathbb{H}_d$:
\begin{equation*}
	\mathbb{H}_d:=\left\{h\in\mathbb{H}: h(2x)\le C h(x) \text{ for some }C>0\right\}.
\end{equation*} 
We will be interested in the special subclass of dimension functions that allow us to classify zero dimensional sets.

\begin{defn}
A function $h\in\HD$ will be called ``zero dimensional dimension function'' if $h\prec x^\alpha$ for any $\alpha>0$. 
We denote by $\mathbb{H}_0$ the subclass of those functions. As a model to keep in mind, consider the family $h_\theta(x)=\frac{1}{\log^\theta(\frac{1}{x})}$.
\end{defn}

Now, given an $\alpha$-dimensional set $E$ but not being an $\alpha$-set, one could expect to find in the class $\mathbb{H}$ an appropriate function $h$ to detect the precise ``size" of it. By that we mean that $0<\HH{h}(E)<\infty$, and in this case $E$ is referred to as an $h$-set. In order to illustrate the main difficulties, we start with a simple observation. The Hausdorff dimension of a set $E\sub\R^n$ is the unique real number $s$ characterized by the following properties:
\begin{itemize}
 \item $\HH{r}(E)=+\infty$ for all $r<s$.
 \item $\HH{t}(E)=0$ for all $s<t$.
\end{itemize}
Therefore, to prove that some set has dimension $s$, it suffices to prove the preceding two properties, and this is independent of the possibles values of $\HH{s}(E)$. It is always true, no matter if $\HH{s}(E)$ is zero, finite and positive, or infinite.

The above observation could lead to the conjecture that in the wider scenario of dimension functions the same kind of reasoning can be made. In fact, Eggleston claims in \cite{egg52} that for any $A\sub\RN$, one of the following three possibilities holds.
\begin{enumerate}
\item for all $h\in\H$, $\HH{h}(A)=0$.
\item there is a function $h_0\in\H$, such that if $h\succ h_0$ then $\HH{h}(A)=0$, whilst if $h\prec h_0$, then $\HH{h}(A)=+\infty$.
\item for all $h\in\H$, $\HH{h}(A)=+\infty$.
\end{enumerate}
Note that the most interesting situation is the one on item $2.$, since it is saying that the correct notion of size for the set $A$ is represented by the function $h_0$. Clearly, this is the case when we are dealing with an $h$-set. However, this claim is false, in the sense that there are situations where none of the above three cases is met. The problems arise from two results due to Besicovitch (see \cite{bes56a} and  \cite{bes56b}, also \cite{rog70} and references therein). The first says that if a set $E$ has null $\HH{h}$-measure for some $h\in\H$, then there exists a function $g\prec h$ such that $\HH{g}(E)=0$. Symmetrically, the second says that if a \emph{compact} set $E$ has non-$\sigma$-finite $\HH{h}$ measure, then there exists a function $g\succ h$ such that $E$ has also non-$\sigma$-finite $\HH{g}$ measure. These results imply that if a \emph{compact} set $E$ satisfies that there exists a function $h_0$ such that $\HH{h}(E)>0$ for any $h\prec h_0$ and $\HH{h}(E)=0$ for any $h\succ h_0$, then it must be the case that $0<\HH{h_0}(E)$ and $E$ has $\sigma$-finite $\HH{h}$-measure.

Consider now the set $\LL$ of Liouville numbers. It is known that this set is dimensionless, which means  that it is not an $h$-set for any $h\in\H$. In that direction, further improvements are due to Elekes and Keleti \cite{ek06}. There the authors prove much more than that there is no exact Hausdorff-dimension function for the set $\mathbb{L}$ of Liouville numbers: they prove that for any translation invariant Borel measure $\mathbb{L}$ is either of measure zero or has non-sigma-finite measure. In addition, it is shown in \cite{or06} that there are two proper nonempty subsets $\LL_0, \LL_\infty\sub\H$  of dimension functions such that $\HH{h}(\LL)=0$ for all $h\in\LL_0$ and $\HH{h}(\LL)=\infty$ for all $h\in\LL_\infty$. It follows that the Liouville numbers $\LL$ must satisfy condition $2.$ in the classification of Eggleston. But suppose that $h_0$ is the claimed dimension function in that case. The discussion in the above paragraph implies that the set  $\LL$ is an $h_0$-set, which is a contradiction. 

Since in the present work we are interested in estimates for the size of general Furstenberg sets, we have to consider dimension functions that are a true \emph{step down} or \emph{step up} from the critical one. The natural generalization of the class of Furstenberg sets to the wider scenario of dimension functions is the following.

\begin{defn}\label{def:furs}
Let $\h$ be a dimension function. A set $E\subseteq\RR$ is a Furstenberg set of type $\h$, or an $F_\h$-set, if for each direction $e\in\s$ there is a line segment $\ell_e$ in the direction of $e$ such that  $\HH{\h}(\ell_e \cap E)>0$. 	
\end{defn}
Note that this hypothesis is stronger than the one used to define the original Furstenberg $F_\alpha$-sets. However, the hypothesis $\dim_H(E\cap\ell_e)\ge \alpha$ is equivalent to $\HH{\beta}(E\cap\ell_e)>0$ for any $\beta$ smaller than $\alpha$. If we use the wider class of dimension functions introduced above, the natural way to define $F_\h$-sets would be to replace the parameters $\beta<\alpha$ with two dimension functions satisfying the relation $h\prec \h$. But requiring  $E\cap\ell_e$ to have positive $\HH{h}$ measure for any $h\prec \h$ implies that it has also positive $\HH{\h}$ measure. It will be useful to introduce also the following subclass of $F_\alpha$:

\begin{defn}
A set $E\subseteq\RR$ is an $F^+_\alpha$-set if for each $e\in\s$ there is a line segment $\ell_e$ such that  $\HH{\alpha}(\ell_e \cap E)>0$. 	
\end{defn}

From the preceding discussion, it follows that there is an unavoidable need to study a notion of ``gap'' between dimension functions. We will show that if $E$ is a set in the class $F_\h$, and $h$ is a dimension function that is {\em much smaller} than  $\h^2$ or $\sqrt{\cdot}\; \h$, then $\HH{h}(E) = \infty$ (Theorem \ref{thm:htoh2} and Theorem \ref{thm:htohsqrt} respectively).  We further exhibit a very {\em small} Furstenberg set $F$ in $F_\h$, for some particular choices of $\h$ and show that for this set, if $\sqrt \cdot \;\h^{3/2}$ is {\em much smaller} than $h$, then $\HH{h}(F) = 0$ (Theorem \ref{thm:sqrth3/2}). This generalizes the result of the classical setting given in \eqref{eq:dim}.

The $\h\to\h^2$ bound strongly depends on the known estimates for the Kakeya maximal operator: for an integrable function $f$ on $\RN$, the Kakeya maximal operator at  scale $\delta$ applied to $f$, $\mathcal{K}_\delta(f):\s^{n-1}\to \R$, is 
\begin{equation*}
	\mathcal{K}_\delta(f)(e)=\sup_{x\in\RN}\frac{1}{|T_e^\delta(x)|}\int_{T_e^\delta(x)}|f(x)|\ dx\qquad e\in\s^{n-1},
\end{equation*}
where $T_e^\delta(x)$ is a $1\times \delta$-tube (by this we mean a tube of length 1 and cross section of radius $\delta$) centred at $x$ in the direction of $e\in\s^{n-1}\subset\RN$. It is well known that in $\RR$ the Kakeya maximal function satisfies the bound (see \cite{wol99b})
\begin{equation}\label{eq:maximal}
\big\|\mathcal{K}_\delta(f)\big\|^2_2\le C \log(\frac{1}{\delta})\|f\|_2^2.
\end{equation}

Our proof of Theorem \ref{thm:htoh2} relies on an optimal use of this estimates for the Kakeya maximal function, exploiting the logarithmic factor in the above bound, which is necessary (see \cite{kei99}), because of the existence of Kakeya sets of zero measure. The other lower bound, which is the relevant bound near the zero dimensional case, depends on some combinatorial arguments that we extended to this general setting. In addition, our techniques allow us to extend the bounds in \eqref{eq:dim} to ``zero dimensional'' classes. At the endpoint $\alpha=0$  we can show that for  $\h\in\H_0$ defined by $\h(x)=\frac{1}{\log(\frac{1}{x})}$, any $F_\h$-set  $E$ must satisfy $\dim_H(E)\ge\frac{1}{2}$.

For the upper bounds the aim is to explicitly exhibit constructions of reasonably small Furstenberg sets. To achieve these optimal constructions, we needed a suited version of Jarn\'ik's theorems on Diophantine Approximation. We exhibit an $F_\h$-set whose dimension function can not be much larger (in terms of logarithmic gaps) than $\sqrt \cdot\; \h^{3/2}$ for the classical case of $\h(x)=x^\alpha$.  We also show in Section \ref{sec:upper} a particular set $E \in F_\h$ for $\h(x)=\frac{1}{\log(\frac{1}{x})}$ satisfying $\dim_H(E)\le\frac{1}{2}$. 

We also consider another related problem, both in the classical and generalized setting. We analyze the role of the dimension of the set of directions in the Furstenberg problem. We consider the class of $F_{\alpha\beta}$ sets, defined in the same way as the $F_\alpha$ class but with the directions taken in a subset $L$ of the unit circle such that $\dim_H(L)\ge \beta$. We are able to prove that if $E$ is any $F_{\alpha\beta}$-set, then
\begin{equation}\label{eq:dim_ab_intro}
\dim_H(E)\ge \max\left\{2\alpha+\beta -1;\frac{\beta}{2}+\alpha  \right\},\qquad \alpha,\beta>0,
\end{equation} 
For the proof of one of the lower bounds we needed estimates for the Kakeya maximal function but for more general measures. The other lower bound uses the $\delta$-\emph{entropy} of the set $L$ of directions, which is the maximal possible cardinality of a $\delta$-separated subset. Our results are proved in the context of the general Hausdorff measures and we obtain \eqref{eq:dim_ab_intro} as a corollary. The only previously known bounds in this setting where for the particular case of $\alpha=1$, $\beta\in(0,1]$ (see \cite{mit02}). The author there obtains that if $E$ is an $A$-Kakeya set (that is, a planar set with a unit line segment in any direction $e\in A$ for a set $A\sub \s$), the $\dim_H(E)\ge 1+\dim_H(A)$ (this is only one of the lower bounds).

This paper is organized as follows: In Section \ref{sec:haus.meas} we provide some extra examples and remarks about Hasudorff measures and dimension functions. In Section \ref{sec:lower} we study the problem of finding lower bounds for the size of generalized Furstenberg sets. In Section \ref{sec:fractal} we study the same problem for a more general class of Furstenberg sets associated to a fractal set of directions. Finally, in Section \ref{sec:upper} we study the upper bounds.

As usual, we will use the notation $A\lesssim B$ to indicate that there is a constant $C>0$ such that $A\le C B$, where the constant is independent of $A$ and $B$. By $A\sim B$ we mean that both $A\lesssim B$ and $B\lesssim A$ hold.

\section{Preliminaries on Hausdorff measures and dimension functions}\label{sec:haus.meas}

In this section we introduce some preliminaries on dimension functions and Hausdorff measures. Moreover, we discuss some additional features of the problem of finding appropriate notions of size for fractal sets. Extra examples are also included. 

\subsection{Dimension Partition}

For a given set $E\sub\R^n$, we introduce the notion of dimension partition (see \cite{chm10}).
\begin{defn}
By the \emph{Dimension Partition }  of a set $E$ we mean a partition of $\H$ into (three) sets: $\mathcal{P}(E)=E_0\cup E_1\cup E_\infty$ with
\begin{itemize}
	\item $E_0=\{h\in\H:\HH{h}(E)=0\}$.
	\item $E_1=\{h\in\H:0<\HH{h}(E)<\infty\}$.
	\item $E_\infty=\{h\in\H:\HH{h}(E)\text{ has non-}\sigma\text{-finite } \HH{h} \text{-measure}\}$.
\end{itemize}
\end{defn}
It is very well known that $E_1$ could be empty, reflecting the dimensionless nature of $E$. A classical example of this phenomenon is the set $\mathbb{L}$ of Liouville numbers. On the other hand, $E_1$ is never empty for an $h$-set, but it is  not easy to determine this partition in the general case. We also remark that it is possible to find non-comparable dimension functions $g,h$ and a set $E$ with the property of being a $g$-set and an $h$-set simultaneously. Consider the following example:

\begin{ex}\label{ex:hg}
There exists a set $E$ and two dimension functions $g,h\in\H$ which are not comparable and such that $E$ is a $g$-set and also an $h$-set.
\end{ex}
\begin{proof}
 We will use the results of \cite{cmms04}. The set $E$ will be the Cantor set $C_a$ associated to a nonnegative decreasing sequence $a=\{a_i\}$ such that $\sum a_i=1$. We start by removing an interval of length $a_1$. Then we remove an interval of length $a_2$ from the left and of length $a_3$ at the right. Following this scheme, we end up with a perfect set of zero measure. If we define $b_n=\frac{1}{n}\sum_{i\ge n}a_i$, then the main result of the cited work is that
 \begin{equation}\label{eq:criterioCa}
  \varliminf_{n\to\infty} nh(b_n)\sim \mathcal{H}^h(C_a),
 \end{equation} 
for all $h\in\H$.
The authors prove that there is possible to construct a spline-type dimension function $h=h_a$ that makes $C_a$ an $h$-set. Further, the function $h$ satisfies that $h(b_n)=\frac{1}{n}$. 
Now we want to define $g$. Consider the sequence $x_n=b_{n!}$ and take $g$ satisfying the following properties:
\begin{enumerate}
	\item $g(x)\ge h(x)$ for all $x>0$.
	\item $g(x_n)= h(x_n)$ for all $n\in\N$,
	\item $g$ is a polygonal spline (same as $h$), but it is constant in each interval $[b_{n!-1},b_{(n-1)!}]$ and drops abruptly on $[b_{n!},b_{n!-1}]$ (we are building up $g$ from the right approaching the origin). More precisely, for each $n\in\N$,
\begin{equation*}
  g(x)=\left\{
    \begin{array}{ccc}
    \frac{1}{(n-1)!}& \mbox{ if } & x\in [b_{n!-1},b_{(n-1)!}]\\
    \frac{1}{n!}& \mbox{ if } & x=b_{n!}\\
    \end{array}\right.
\end{equation*}
and it is linear on $[b_{n!},b_{n!-1}]$.
\end{enumerate}
Conditions 1 and 2 imply that $\varlimsup_{x\to 0} \frac{h(x)}{g(x)}=1<\infty $. Note that we also have that 
$\varliminf_{x\to 0} \frac{h(x)}{g(x)}=0 $, since 
\begin{equation*}
\frac{h(b_{n!-1})}{g(b_{n!-1})}=\frac{(n-1)!}{n!-1}\sim\frac1{n}\to0.
\end{equation*}
It follows that $h$ and $g$ are not comparable. To see that $C_a$ is also a $g$-set, we use again the characterization \eqref{eq:criterioCa}. Since
\begin{equation*}
\varliminf_{n\to\infty} ng(b_{n})\le\varliminf_{n\to\infty} n!g(b_{n!})=\varliminf_{n\to\infty} n!h(b_{n!})<\infty,
\end{equation*}
we obtain that $\mathcal{H}^g(C_a)<\infty$. In addition, $g(x)\ge h(x)$ for all $x$, hence $g(b_n)\ge h(b_n)$ for all $n\in \N$ and it follows that
\begin{equation*}
\varliminf_{n\to\infty} ng(b_{n})\ge\varliminf_{n\to\infty} nh(b_{n})>0.
\end{equation*}
and therefore $\mathcal{H}^g(C_a)>0$. 
\end{proof}
We refer the reader to \cite{gms07} for a detailed study of the problem of equivalence between dimension functions and Cantor sets associated to sequences. The authors also study Packing measures and premeasures of those sets. 
For the construction of $h$-sets associated to certain sequences see the work of Cabrelli {\textit{et al}} \cite{cmms04}. 

It follows from Example \ref{ex:hg} that even for $h$-sets the dimension partition, and in particular $E_1$, is not completely determined. Note that the results of Rogers cited above imply that, for compact sets, $E_0$ and $E_\infty$ can be thought of as open components of the partition, and $E_1$ as the ``border'' of these open components.
An interesting problem is then to determine some criteria to classify the functions in $\H$ into those classes. To detect where this ``border'' is, we will introduce the notion of \emph{chains} in $\H$. This notion allows to refine the notion of Hausdorff dimension by using an ordered family of dimension functions. More precisely, we have the following definition.

\begin{defn}
A family $\mathcal{C}\subset\H$ of dimension functions will be called a \emph{chain} if it is of the form                                   
\begin{equation*}
\mathcal{C}=\left\{h_t\in\H: t\in\R, h_s\prec h_t \iff s<t\right\}
\end{equation*}  
That is, a totally ordered one-parameter family of dimension functions.
\end{defn}

Suppose that $h\in\H$ belongs to some chain $\mathcal{C}$ and satisfies that, for any $g\in\mathcal{C}$, $\HH{g}(E)>0$ if $g\prec h$ and $\HH{g}(E)=0$ if $g\succ h$. Then, even if $h\notin E_1$, in this chain, $h$ does measure the size of $E$. It can be thought of as being ``near the frontier'' of both $E_0$ and $E_\infty$. For example, if a set $E$ has Hausdorff dimension $\alpha$ but $\HH{\alpha}(E)=0$ or $\HH{\alpha}(E)=\infty$, take $h(x)=x^\alpha$ and $\mathcal{C}_H=\{x^t:t\ge 0\}$. In this chain, $x^\alpha$ is the function that best measures the size of $E$.

We look for finer estimates, considering chains of dimension functions that yield ``the same Hausdorff dimension''. Further, for zero dimensional sets, this approach allows us to classify them by some notion of dimensionality.

\subsection{The exact dimension function for a class of sets}

In the previous section we dealt with the problem of detecting an appropriate dimension function for a given set or, more generally, the problem of determining the dimension partition of that set. Now we introduce another related problem, which concerns the analogous problem but for a whole class of sets defined, in general, by geometric properties.
We mention one example: As we mention before, all Kakeya sets in $\RR$ have full dimension, but even in that case, there are several distinct types of 2-dimensional sets (for instance, with positive or null Lebesgue measure). Hence, one would like to associate a dimension function to the whole class. A dimension function $h\in\mathbb{H}$ will be called the exact Hausdorff dimension function of the class of sets $\mathcal{A}$ if 
\begin{itemize}
\item For every set $E$ in the class $\mathcal{A}$, $\mathcal{H}^h(E)>0$.
\item There are sets $E\in\mathcal{A}$ with $\mathcal{H}^h(E)<\infty$.
\end{itemize}

In the direction of finding the exact dimension of the class of Kakeya sets in $\RR$, Keich  has proven in \cite{kei99} that the exact dimension function $h$ must decrease to zero at the origin faster than $x^2\log(\frac{1}{x})\log\log(\frac{1}{x})^{2+\e}$ for any given $\e>0$, but slower than $x^2\log(\frac{1}{x})$. This notion of speed of convergence tells us precisely that $h$ is between those two dimension functions (see Definition \ref{def:order}). More precisely, the author explicitly construct a small Kakeya set, which is small enough to have finite $g$ measure for  $g(x)=x^2\log(\frac{1}{x})$. Therefore, for $h$ to be an exact dimension function for the class of Kakeya sets, it cannot be dimensionally greater than $g$. But this last condition is not sufficient to ensure that \emph{any} Kakeya set has positive $h$-measure. The partial result from \cite{kei99} is that for any $\e>0$ and any Kakeya set $E$, we have that $\HH{h_\e}(E)>0$, where $h_\e=x^2\log(\frac{1}{x})\log\log(\frac{1}{x})^{2+\e}$.

\section{Lower bounds for Furstenberg sets}\label{sec:lower}

In this section we deal with the problem of finding sharp lower bounds for the generalized dimension of Furstenberg type sets. Let us begin with some remarks about this problem and the techniques involved. 

\subsection{Techniques} 

We start with a \emph{uniformization} procedure. Given an $F_\h$-set $E$ for some $\h\in\mathbb{H}$, it is always possible to find two constants $m_E,\delta_E>0$ and a set $\Omega_E\sub\s$ of positive $\sigma$-measure such that
	\begin{equation*}
			\hh{\h}(\ell_e\cap E)>m_E>0 \qquad \forall\delta<\delta_E\quad,\quad \forall e\in\Omega_E.
	\end{equation*} 
For each $e\in\s$, there is a positive constant $m_e$ such that $\HH{\h}(\ell_e\cap E)>m_e$. Now consider the following pigeonholing argument. Let $\Lambda_n=\{e\in\s: \frac{1}{n+1}\le m_e<\frac{1}{n}\}$. At least one of the sets must have positive measure, since $\s=\cup_n \Lambda_n $. Let $\Lambda_{n_0}$ be such set and take $0<2m_E<\frac{1}{n_0+1}$. Hence $\HH{\h}(\ell_e\cap E)>2m_E>0$ for all $e\in\Lambda_{n_0}$. 
Finally, again by pigeonholing, we can find $\Omega_E\subseteq\Lambda_{n_0}$ of positive measure and $\delta_E>0$ such that
\begin{equation}\label{eq:hypot}
	\hh{\h}(\ell_e\cap E)>m_E>0\qquad \forall e\in\Omega_E \quad \forall \delta<\delta_E.
\end{equation} 

To simplify notation throughout the remainder of the chapter, since inequality \eqref{eq:hypot} holds for any Furstenberg set and we will only use the fact that $m_E$, $\delta_E$ and $\sigma(\Omega_E)$ are positive, it will be enough to consider the following definition of $F_\h$-sets:
\begin{defn}\label{def:general}
Let $\h$ be a dimension function. A set $E\subseteq\RR$ is Furstenberg set of type $\h$, or an $F_\h$-set, if for each $e\in\s$ there is a line segment $\ell_e$ in the direction of $e$ such that  $\hh{\h}(\ell_e \cap E)>1$ for all $\delta<\delta_E$ for some $\delta_E>0$.
\end{defn}

The following technique is a standard procedure in this area. The lower bounds for the Hausdorff dimension of a given set $E$, both in the classical and general setting, are achieved by bounding uniformly from below the size of the coverings of $E$. More precisely, the $h$-size of a covering $\mathcal{B}=\{B_j\}$ is $\sum_j h(r_j)$. Our aim will be then to prove essentially that $\sum_j h(r_j)\gtrsim 1$, provided that $h$ is a small enough dimension function.  We introduce the following notation:
\begin{defn}\label{def:bk}
Let $\mathfrak{b}=\{b_k\}_{k\in\N}$ be a decreasing sequence with $\lim b_k=0$. For any family of balls $\mathcal{B}=\{B_j\}$ with $B_j=B(x_j;r_j)$, $r_j\le 1$, and for any set $E$, we define
\begin{equation}\label{eq:jkb}
J^\b_k:=\{j\in\N:b_k< r_j\le b_{k-1}\},
\end{equation}
and 
\begin{equation*}
E_k:=E\cap \bigcup_{j\in J^\mathfrak{b}_k}B_j.
\end{equation*}
In the particular case of the dyadic scale $\mathfrak{b}=\{2^{-k}\}$, we will omit the superscript and denote
\begin{equation*}
J_{k}:=\{j\in\N:2^{-k}<r_{j}\le2^{-k+1}\}.
\end{equation*} 
\end{defn}
The idea will be to use the dyadic partition of the covering to obtain that $	\sum_{j\ge0} h(r_j)\gtrsim \sum_{k\ge0} h(2^{-k})\#J_k$.
The lower bounds we need will be obtained if we can prove lower bounds on the quantity $J_k$ in terms of the function $h$ but independent of the covering. The next lemma introduces a technique we borrow from \cite{wol99b} to decompose  the set of all directions. 
\begin{lem}\label{lem:part}
	Let $E$ be an $F_\h$-set for some $\h\in\mathbb{H}$ and $\a=\{a_k\}_{k\in\N}\in\ell^1$ a non-negative sequence. Let $\mathcal{B}=\{B_j\}$ be a $\delta$-covering of $E$ with $\delta< \delta_E$ and let $E_k$ and $J_k$ be as above. Define
\begin{equation*}
	\Omega_k:=\left\{e\in\s: \hh{\h}(\ell_e\cap E_k)\ge \frac{a_k}{2\|\a\|_1}\right\}.
\end{equation*}
Then $\s=\cup_k \Omega_k$.
\end{lem}
\begin{proof}
It follows directly form the summability of $\a$.
\end{proof}
				  We will need in the next section the main result of \cite{mit02}, which is the following proposition.
\begin{prop}\label{pro:maximalgeneral}
Let $\mu$ be a Borel probability measure on $\s$ such that $\mu(B(x,r))\lesssim \varphi(r)$ for some non-negative function $\varphi$ for all $r\ll 1$. Define the Kakeya maximal operator $\mathcal{K}_\delta$ as usual:
\[
\mathcal{K}_\delta(f)(e)=\sup_{x\in\RN}\frac{1}{|T_e^\delta(x)|}\int_{T_e^\delta(x)}|f(x)|\ dx,\qquad e\in\s^{n-1}.
\]
Then we have the estimate
\begin{equation}\label{eq:maximalgeneral}
 \|\mathcal{K}_\delta\|^2_{L^2(\RR)\to L^2(\s,d\mu)}\lesssim C(\delta)=\int_\delta^1\frac{\varphi(u)}{u^2}du.
\end{equation}
\end{prop}

\begin{rem}
 It should be noted that if we choose $\varphi(x)=x^s$, then we obtain as a corollary that 
\begin{equation*}
 \|\mathcal{K}_\delta\|^2_{L^2(\RR)\to L^2(\s,d\mu)}\lesssim \delta^{s-1}.
\end{equation*}
In the special case of $s=1$, the bound has the known logarithmic growth:
\[
 \|\mathcal{K}_\delta\|^2_{L^2(\RR)\to L^2(\s,d\mu)}\sim \log(\frac{1}{\delta}).
\]
\end{rem}

\subsection{The \texorpdfstring{$\h\to \h^2$}{h-h2} bound }\label{sec:h2}

In this section we generalize the first inequality of \eqref{eq:dim}, that is, $\dim_H(E)\ge 2\alpha$ for any $F_\alpha$-set. For this, given a dimension function $h\prec \h^2$, we impose some sufficient growth conditions on the gap $\frac{\h^2}{h}$ to ensure that $\HH{h}(E)>0$. We have the following theorem:

\begin{thm}\label{thm:htoh2}
	Let $\h\in\mathbb{H}_d$ be a dimension function and let $E$ be an $F_\h$-set. Let $h\in\mathbb{H}$  such that $h\prec \h^2$. If ${\D\sum_{k\ge0}} \sqrt{k\frac{\h^2}{h}(2^{-k})}<\infty$, then $\HH{h}(E)>0$.
\end{thm}

\begin{proof}
By Definition \ref{def:general}, since $E\in F_\h$, we have $\hh{\h}(\ell_e\cap E)>1$ for all $e\in\s$ and for any $\delta<\delta_E$. Let $\{B_j\}_{j\in\N}$ be a covering of $E$ by balls with $B_j=B(x_j;r_j)$. We need to bound $\sum_j h(2r_j)$ from below. Since $h$ is non-decreasing, it suffices to obtain the  bound $\sum_j h(r_j)\gtrsim 1$ for any $h\in \mathbb{H}$ satisfying the hypothesis of the theorem. Clearly we can restrict ourselves to $\delta$-coverings with $\delta<\frac{\delta_E}{5}$. Define $\a=\{a_k\}$ with $a_k=\sqrt{k\frac{\h^2}{h}(2^{-k})}$. By hypothesis, $\a\in\ell^1$. Also define, as in the previous section, for each $k\in\N$, $J_k=\{j\in \N: 2^{-k}< r_j\le 2^{-k+1}\}$ and $E_k=E\cap\D\cup_{j\in J_k}B_j$.
Since $\a\in\ell^1$, we can apply Lemma \ref{lem:part} to obtain the decomposition $\s=\bigcup_k \Omega_k$ associated to this choice of $\a$.

We will apply the maximal function inequality to a weighted union of indicator functions. For each $k$, let $F_k=\D\bigcup_{j\in J_k}B_j$ and define the function 
\[
f:=\h(2^{-k})2^k\D\chi_{F_k}.
\]

We will use the $L^2$ norm estimates for the maximal function.
The $L^2$ norm of $f$ can be easily estimated as follows:
\begin{equation*}
\|f\|_{2}^{2} = \h^2(2^{-k})2^{2k}\int_{\cup_{J_k}B_{j}}dx \lesssim \h^2(2^{-k})2^{2k}\sum_{j\in J_{k}}r_j^2 \lesssim\h^2(2^{-k}) \#J_k,
\end{equation*}
since $r_j\le 2^{-k+1}$ for $j\in J_k$.
Hence, 
\begin{equation}\label{eq:L2above}
	\|f\|_{2}^{2}\lesssim \#J_k \h^2(2^{-k}).
\end{equation} 
Now fix $k$ and consider the Kakeya maximal function $\mathcal{K}_\delta(f)$ of level $\delta=2^{-k+1}$ associated to the function $f$ defined for this value of $k$.

In $\Omega_k$ we have the following pointwise lower estimate for the maximal function. Let $\ell_e$ be the line segment such that $\hh{\h}(\ell_e\cap E)>1$, and let $T_e$ be the rectangle of width $2^{-k+2}$ around this segment. 
Define, for each $e\in\Omega_k$, 
\begin{equation*}
 J_k(e):=\{j\in J_k: \ell_e\cap E\cap B_j\neq\emptyset\}.
\end{equation*}

With the aid of the Vitali covering lemma, we can select a subset of disjoint balls $\widetilde{J}_k(e)\sub J_k(e)$ such that 
\begin{equation*}
	\bigcup_{j\in J_k(e)}B_j \sub \bigcup_{j\in\widetilde{J}_k(e)}B(x_j;5r_j).
\end{equation*} 

Note that every ball $B_j$, $j\in J_k(e)$, intersects $\ell_e$ and therefore at least half of $B_j$ is contained in the rectangle $T_e$, yielding $|T_e\cap B_j|\ge\frac{1}{2}\pi r_j^2$. Hence, by definition of the maximal function, using that $r_j\ge 2^{-k+1}$ for $j\in J_k(e)$,
\begin{eqnarray*}
|\mathcal{K}_{2^{-k+1}}(f)(e)| 	& \geq 	& \frac{1}{|T_e|}\int_{T_e}f\ dx =\frac{\h(2^{-k})2^{k}}{|T_{e}|}\left|T_e\cap \cup_{J_k(e)}B_{j}\right|\\
	& \gtrsim &\h(2^{-k})2^{2k} \left|T_e\cap \cup_{\widetilde{J}_k(e)}B_{j}\right|\\
	& \gtrsim & \h(2^{-k})2^{2k}\sum_{j\in \widetilde{J}_{k}(e)}r^2_j\\
	&\gtrsim &\h(2^{-k})\#\widetilde{J}_{k}(e)\\
	& \gtrsim & \sum_{\widetilde{J}_{k}(e)}\h(r_j).
    \end{eqnarray*}
Now, since
\begin{equation*}
	\ell_e\cap E_k\sub \bigcup_{j\in J_k(e)}B_j\sub \bigcup_{j\in \widetilde{J}_k(e)}B(x_j;5r_j)
\end{equation*}
and for $e\in\Omega_k$ we have $\hh{\h}(\ell_e\cap E_k)\gtrsim a_k$, we obtain 
\begin{equation*}
	|\mathcal{K}_{2^{-k+1}}(f)(e)|
	\gtrsim\sum_{\widetilde{J}_{k}(e)}\h(r_j)
	\gtrsim\sum_{j\in \widetilde{J}_{k}(e)}\h(5r_j)\gtrsim a_k.
\end{equation*} 
Therefore we have the estimate
\begin{equation}\label{eq:maximalbelow}
	\|\mathcal{K}_{2^{-k+1}}(f)\|_2^2\gtrsim\int_{\Omega_k} |f^{*}_{2^{-k+1}}(e)|^2\ d\sigma \gtrsim a_k^2\ \sigma(\Omega_k)=\sigma(\Omega_k)k\frac{\h^2}{h}(2^{-k}).
\end{equation} 
Combining \eqref{eq:L2above},  \eqref{eq:maximalbelow} and using the maximal inequality \eqref{eq:maximal}, we obtain
\begin{equation*}
	\sigma(\Omega_k)k\frac{\h^2}{h}(2^{-k})\lesssim \|f^{*}_{2^{-k+1}}\|_2^2\lesssim \log(2^k)\|f\|_2^2\lesssim k\#J_k \h^2(2^{-k}).
\end{equation*} 
Now let $h$ be a dimension function satisfying the hypothesis of Theorem \ref{thm:htoh2}. We have
\begin{equation*}
	\sum_{j\ge0} h(r_j) \ge \sum_{k\ge0} h(2^{-k})\#J_k \gtrsim \sum_{k\ge0} \sigma(\Omega_k) \ge\sigma(\s)>0.
\end{equation*}
\end{proof}

Applying this theorem to the class $F^+_\alpha$, we obtain a sharper lower bound on the generalized Hausdorff dimension:

\begin{cor}\label{cor:coro1}
	Let $E$ an $F^+_\alpha$-set. If $h$ is any dimension function satisfying the relation $h(x)\ge C x^{2\alpha} \log^{1+\theta}(\frac{1}{x})$ for $\theta>2$ then $\HH{h}(E)>0$.
\end{cor}

\begin{rem}
	At the endpoint $\alpha=1$, this estimate is worse than the one due to Keich. He obtained, using strongly the full dimension of a ball in $\RR$, that if $E$ is an $F^+_1$-set and $h$ is a dimension function satisfying the bound $h(x)\ge C x^2 \log(\frac{1}{x})\left(\log\log(\frac{1}{x})\right)^{\theta}$  for  $\theta>2$, then $\HH{h}(E)>0$.
\end{rem}

\begin{rem}
Note that the proof above relies essentially on the $L^1$ and  $L^2$ size of the ball in $\RR$, not on the dimension function $\h$. Moreover, we only use the ``gap" between $h$ and $\h^2$ (measured by the function $\frac{\h^2}{h}$). This last observation leads to conjecture that this proof can not be used to prove that an $F_\h$-set has positive $\h^2$ measure, since in the case of  $\h(x)=x$, as we remarked in the introduction, this would contradict the existence of Kakeya sets of zero measure in $\RR$.

Also note that the absence of conditions on the function $\h$ allows us to consider the ``zero dimensional'' Furstenberg problem. However, this bound does not provide any substantial improvement, since the zero dimensionality property of the function $\h$ is shared by the function $\h^2$. This is because the proof above, in the case of the $F_\alpha$-sets, gives the worse bound ($\dim_H(E)\ge 2\alpha$) when the parameter $\alpha$ is in $(0,\frac{1}{2})$. 
\end{rem}

\subsection{The \texorpdfstring{$\h\to \h\sqrt{\cdot}$}{h-hsqrt} bound, positive dimension}\label{sec:hsqrt+}

Now we will turn our attention to those functions $h$ that satisfy the bound $h(x)\lesssim x^\alpha$ for $\alpha\le\frac{1}{2}$. For these functions we are able to improve on the previously obtained bounds. We need to impose some growth conditions on the dimension function $\h$. This conditions can be thought  of as imposing a lower bound on the dimensionality of $\h$ to keep it away from the zero dimensional case. 

The next lemma is from \cite{mr10} and says that we can split the $\h$-dimensional mass of a set $E$ contained in an interval $I$ into two sets that are positively separated.

\begin{lem}\label{lem:split}
 Let $\h\in\mathbb{H}$, $\delta>0$, $I$ an interval and $E\sub I$. Let $\eta>0$ be such that $\h^{-1}(\frac{\eta}{8})<\delta$ and $\hh{\h}(E)\ge\eta>0$. Then there exist two subintervals $I^-$, $I^+$ that are $\h^{-1}(\frac{\eta}{8})$-separated and  with $\hh{\h}(I^{\pm}\cap E)\gtrsim \eta$.
\end{lem}

The key geometric ingredient is contained in the following lemma. The idea is from \cite{wol99b}, but the general version needed here is from  \cite{mr10}. This lemma will provide an estimate for the number of lines with certain separation property that intersect two balls of a given size.

\begin{lem}\label{lem:conobolas}													
Let $\mathfrak{b}=\{b_k\}_{k\in\N}$ be a decreasing sequence with $\lim b_k=0$. Given a family of balls $\mathcal{B}=\{B(x_j;r_j)\}$, we define $J^\b_k$ as in \eqref{eq:jkb} and let $\{e_i\}_{i=1}^{M_k}$ be a $b_k$-separated set of directions. Assume that for each $i$ there are two line segments $I_{e_i}^+$ and $I_{e_i}^-$ lying on a line in the direction $e_i$  that are $s_k$-separated for some given $s_k$ 
Define $\Pi_k=J^\b_{k}\times
J^\b_{k}\times\{1,..,M_k\}$ and $\mathcal{L}^\b_k$ by

\begin{equation*}
\mathcal{L}^\b_k:=\left\{(j_{+},j_{-},i)\in \Pi_k:
    	I_{e_{i}}^-\cap B_{j_-}\neq\emptyset\ 
	I_{e_{i}}^+\cap B_{j_+}\neq\emptyset
 \right\}.
\end{equation*} 
If $\frac{1}{5}s_k>b_{k-1}$ for all $k$, then
\begin{equation*}
\#\mathcal{L}^\b_k\lesssim\frac{b_{k-1}}{b_k}\frac{1}{s_k}\left(\#J^\b_{k}\right)^{2}.
\end{equation*} 
\end{lem}

\begin{proof}
Consider a fixed pair $j_-,j_+$ and its associated $B_{j_-}$ and $B_{j_+}$ 
We will use as distance between two balls the distance between the centres, and for simplicity we denote $d(j_-,j_+)=d(B_{j_-},B_{j_+})$. If $d(j_-,j_+)<\frac{3}{5}s_k$ then there is no $i$ such that $(j_-,j_+,i)$ belongs to $\mathcal{L}^\b_k$.

Now, for  $d(j_-,j_+)\ge\frac{3}{5}s_k$, we will look at the special configuration given by Figure \ref{fig:conobolas} when we have $r_{j_-}=r_{j_+}=b_{k-1}$ and the balls are tangent to the ends of $I^-$ and $I^+$. This will give a bound for any possible configuration, since in any other situation the cone of allowable directions is narrower.

\begin{figure}[ht]
\begin{tikzpicture}[scale=.75]
\path (0,0) coordinate (O);
\path (O) ++(4,0) coordinate (I+R);
\path (O) ++(2.5,0) coordinate (I+L);
\path (O) ++(-4,0) coordinate (I-L);
\path (O) ++(-2.5,0) coordinate (I-R);
\draw[dotted] (I-L) -- (I+R);
\draw[thick] (I-L)node[below]{$I^-$} -- (I-R);
\draw[thick] (I+L) -- (I+R)node[below]{$I^+$};
\path (O) ++(-1.5,.1) coordinate (1up);
\path (O) ++(-.5,.1) coordinate (2up);
\path (O) ++(.5,.1) coordinate (3up);
\path (O) ++(1.5,.1) coordinate (4up);
\path (O) ++(-1.5,-.1) coordinate (1do);
\path (O) ++(-.5,-.1) coordinate (2do);
\path (O) ++(.5,-.1) coordinate (3do);
\path (O) ++(1.5,-.1) coordinate (4do);

\path (4up) ++(0,-.1) coordinate (4);
\path (1up) ++(0,-.1) coordinate (1);
\path (4) ++(0,1) coordinate (radio);

\draw (1up) -- (1do);
\draw (2up) -- (2do);
\draw (3up) -- (3do);
\draw (4up) -- (4do);

\path (O) ++(42.5:3) coordinate (Cupright);
\path (O) ++(-42.5:3) coordinate (Cdownright);
\path (O) ++(137.5:3) coordinate (Cupleft);
\path (O) ++(222.5:3) coordinate (Cdownleft);

\draw (O) -- (Cupright);
\draw (O) -- (Cdownright);
\draw (O) -- (Cupleft);
\draw (O) -- (Cdownleft);

\draw(4) circle (1);
\draw(1) circle (1);
\draw[dotted] (4) --node[pos=.5, right]{$b_{k-1}$} (radio) ;
\path (1) ++ (135:1.5) node{$B_{j_{-}}$};
\path (4) ++ (-45:1.5) node{$B_{j_{-}}$};
\path (I-R) ++ (0,-1.5) coordinate (dl);
\path (I+L) ++ (0,-1.5) coordinate (dr);
\draw[|<->|] (dl) -- node[below] {$s_{k}$} (dr);
\end{tikzpicture}
\caption{Cone of allowable directions I}
\label{fig:conobolas}
\end{figure}
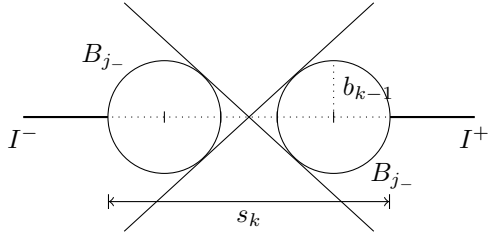

Let us focus on one half of the cone (Figure \ref{fig:cono}). Let $\theta$ be the width of the cone. In this case, we have to look at  $\frac{\theta}{b_k}$ directions that are $b_k$-separated. Further, we note that $\theta=\frac{2\theta_k}{s_k}$, where $\theta_k$ is the bold arc at distance $s_k/2$ from the center of the cone. Let us see that $\theta_k\sim b_{k-1}$. 

\begin{figure}[ht]
\begin{tikzpicture}[scale=.7]
\path (0,0) coordinate (O);
\path (O) ++(6,0) coordinate (1)node[below right]{$1$};
\path (O) ++(30:6) coordinate (up);
\path (O) ++(-30:6) coordinate (down);
\draw (O) -- (up);
\draw (O) -- (down);

\path (O) ++(0:2.62) coordinate (sk/2);
\draw (O) -- (1);
\draw[thick] (O) -- node [below]{$\frac{s_{k}}{2}$} (sk/2);
\path (O) ++(0:1.75) coordinate (center);
\path (center) ++(0,.87) coordinate (radio);
\draw (center) circle (.87);
\draw (center) --node[pos=.5, right]{$b_{k-1}$} (radio);
\draw[thick] (0:2.62) arc (0:30:2.62);
\path (O) ++(15:3) coordinate (thetak) node{$\theta_{k}$};
\draw (-30:6) arc (-30:30:6);
\path (O) ++(15:6.5) coordinate (theta) node{$\theta$};

\end{tikzpicture}
\caption{Cone of allowable directions II}
\label{fig:cono}
\end{figure}
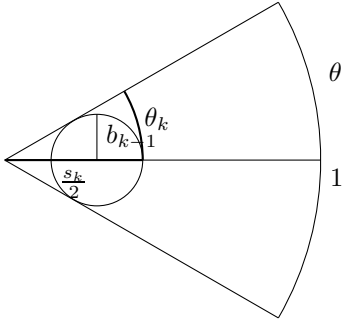

If we use the notation of Figure \ref{fig:arcotang}, we have to prove that $\theta_k\lesssim b_{k-1}$ for $a\in(0,+\infty)$. 
We have $\theta_k=\theta(a+2b_{k-1})$. Also $\theta<\tan^{-1}(\frac{b_{k-1}}{a})$, so 
\begin{equation*}
 \theta_k <\tan^{-1}(\frac{b_{k-1}}{a})(a+2b_{k-1})\sim b_{k-1}.
\end{equation*}

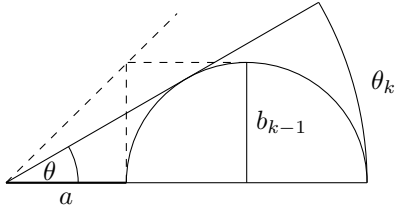
\begin{figure}[ht]
\begin{tikzpicture}[scale=.8]
\path (0,0) coordinate (O);
\path (O) ++(6,0) coordinate (1);
\path (O) ++(2,0) coordinate (a);
\draw (O) -- (1);
\path (O) ++(4,0) coordinate (center);
\path (O) ++(30:6) coordinate (up);
\draw (O) -- (up);
\draw (1) arc (0:180:2);
\draw[thick] (O) -- node[below]{$a$} (a);
\path (a) ++(0,2) coordinate (aup);
\path (center) ++(0,2) coordinate (radio);
\draw(center) -- node [right]{$b_{k-1}$}(radio);
\draw[dashed](a) -- (aup);
\draw[dashed](radio) -- (aup);
\path (O) ++(44.7:4) coordinate (upper);
\draw[dashed] (O) -- (upper);
\draw (1) arc (0:30:6);
\path (O) ++(15:6.5) coordinate (thetak)node{$\theta_{k}$};
\path (O) ++(1.2,0) coordinate (theta);
\path (O) ++(15:.75) coordinate (thetak)node{$\theta$};
\draw (theta) arc (0:30:1.2);

\end{tikzpicture}
\caption{The arc $\theta_k$ is comparable to $b_{k-1}$}\label{fig:arcotang}
\end{figure}

We conclude that $\theta_k\sim b_{k-1}$, and therefore the number $D$ of lines in $b_k$-separated directions with non-empty intersection with $B_{j_-}$ and $B_{j_+}$ has to satisfy $D\le \frac{\theta}{b_k}=\frac{2\theta_k}{s_kb_k}\sim\frac{b_{k-1}}{b_k}\frac{1}{s_k}$. The lemma follows by summing on all pairs $(j_-,j_+)$.
\end{proof}

Now we can present the main result of this section. 

\begin{thm}\label{thm:htohsqrt}
	Let $\h\in\mathbb{H}_d$ be a dimension function such that $\h(x)\lesssim x^\alpha$ for some $0<\alpha<1$ and $E$ be an $F_\h$-set. Let $h\in\mathbb{H}$ with $h\prec \h$.  If ${\D\sum_{k\ge0}} \frac{\h}{h}(2^{-k})^{\frac{2\alpha}{2\alpha+1}}<\infty$, then $\HH{h\sqrt{\cdot}}(E)>0$.
\end{thm}

\begin{proof}

We begin in the same way as in the previous section. Again by Definition \ref{def:general}, since $E\in F_\h$, we have $\hh{\h}(\ell_e\cap E)>1$
for all $e\in\s$ for any $\delta<\delta_E$. 

Consider the sequence $\a=\left\{\frac{\h}{h}(2^{-k})^{\frac{2\alpha}{2\alpha+1}}\right\}_{k}$. Let $k_0$ be such that
\begin{equation}\label{eq:hiplemasplit}
\h^{-1}\left(\dfrac{a_k}{16\|\a\|_1}\right)<\delta_E \qquad \text{for any }k\ge k_0.
\end{equation} 
Now take any $\delta$-covering $\mathcal{B}=\{B_j\}$ of $E$ by balls with $\delta<\min\{\delta_E, 2^{-k_0}\}$. Using Lemma \ref{lem:part} we obtain $\s=\bigcup_k\Omega_k$ with 
\begin{equation}\label{eq:mass}
	\Omega_k=\left\{e\in\Omega: \hh{\h}(\ell_e\cap E_k)\ge \frac{ a_k}{2\|\a\|_1}\right\}.
\end{equation}
Again we have  $E_k=E\cap\bigcup_{j\in J_k}B_j$, but by our choice of $\delta$, the sets $E_k$ are empty for $k<k_0$. Therefore the same holds trivially for $\Omega_k$ and we have that  $\s=\bigcup_{k\ge k_0}\Omega_k$. Since for each $e\in\Omega_k$ we have  the inequality in \eqref{eq:mass}, we can apply Lemma \ref{lem:split} with $\eta=\frac{a_k}{2\|\a\|_1}$ to $\ell_e\cap E_k$. Therefore we obtain two intervals $I_e^-$ and $I_e^+$, contained in $\ell_e$ with
\begin{equation*}
	\hh{\h}(I^{\pm}_e\cap E_k)\gtrsim a_k
\end{equation*} 
that are $\h^{-1}(ra_k)$-separated for $r=\frac{1}{16\|\a\|_1}$. Let $\{ e^k_{j}\}_{j=1}^{M_k}$ be a $2^{-k}$-separated subset of $\Omega_k$. Therefore $M_k\gtrsim2^{k}\sigma(\Omega_k)$. Define $\Pi_k:=J_{k}\times
J_{k}\times\{1,..,M_k\}$ and 
\begin{equation*}
\mathcal{T}_k:=\left\{(j_{-},j_{+},i)\in \Pi_k:
    	I_{e_{i}}^-\cap E_{k}\cap B_{j_-}\neq\emptyset\ 
	I_{e_{i}}^+\cap E_{k}\cap B_{j_+}\neq\emptyset
 \right\}.
\end{equation*} 
We will count the elements of $\mathcal{T}_k$ in two different ways. First, fix $j_{-}$ and $j_{+}$ and count for how many values of $i$ the triplet $(j_{-},j_{+},i)$ belongs to $\mathcal{T}_k$. For this, we will apply Lemma \ref{lem:conobolas} for the choice $\b=\{2^{-k}\}$. The estimate we obtain is the number of $2^{-k}$-separated directions $e_i$, that intersect simultaneously the balls $B_{j_-}$ and $B_{j_+}$, given that these balls are separated. We obtain
\begin{equation*}
\#\mathcal{T}_k\lesssim\frac{1}{\h^{-1}(ra_k)}\left(\#J_{k}\right)^{2}.
\end{equation*} 
Second, fix $i$. In this case, we have by hypothesis that  $\hh{\h}(I_{e_i}^{+}\cap E_{k})\gtrsim a_k$, so $\sum_{j_{+}}\h(r_{j_+})\gtrsim a_k$. Therefore, 
\begin{equation*}
 a_k\lesssim\sum_{(j_-,j_+,i)\in\mathcal{T}_k}\h(r_{j_+})\le K \h(2^{-k}),
\end{equation*} 
where $K$ is the number of elements of the sum. Therefore $K\gtrsim \frac{a_k}{\h(2^{-k})}$. The same holds for $j_{-}$, so
\begin{equation*}
\#\mathcal{T}_k\gtrsim M_k\left(\frac{a_k}{\h(2^{-k})}\right)^2.
\end{equation*} 
Combining the two bounds,
\begin{equation*}
\#J_{k}  \gtrsim  M_k^{1/2}\frac{a_k}{\h(2^{-k})}\h^{-1}(ra_k)^{1/2} \gtrsim 2^{\frac{k}{2}}\sigma(\Omega_k)^{1/2}\frac{a_k}{\h(2^{-k})}\h^{-1}(ra_k)^{1/2}.
\end{equation*}
Consider now a dimension function $h\prec \h$ as in the hypothesis of the theorem. Then again

\begin{equation}\label{eq:cotahaushsqrt}
\sum_j h(r_j)r_j^{1/2} \ge \sum_{k} \frac{\h(2^{-k})2^{-\frac{k}{2}}\#J_k}{\frac{\h}{h}(2^{-k})}\gtrsim\sum_{k\ge k_0} \sigma(\Omega_k)^{1/2}\frac{a_k\h^{-1}(ra_k)^{1/2}}{\frac{\h}{h}(2^{-k})}.
 \end{equation} 
To bound this last expression, we use first that there exists $\alpha\in(0,1)$ with $\h(x)\lesssim x^\alpha$ and therefore $\h^{-1}(x)\gtrsim x^{\frac{1}{\alpha}}$. We then recall the definition of the sequence $\a$,  $a_k=\frac{\h}{h}(2^{-k})^{\frac{2\alpha}{2\alpha+1}}$ to obtain
\begin{equation*}
\sum_j h(r_j)r_j^{1/2} \gtrsim \sum_{k\ge k_0} \sigma(\Omega_k)^{1/2}\frac{a_k^{\frac{2\alpha+1}{2\alpha}}}{\frac{\h}{h}(2^{-k})}
= \sum_{k\ge k_0}\sigma(\Omega_k)^{1/2}\gtrsim 1.
\end{equation*}
\end{proof}

The next corollary follows from Theorem \ref{thm:htohsqrt} in the same way as Corollary \ref{cor:coro1} follows from Theorem \ref{thm:htoh2}.
\begin{cor}\label{cor:coro2}
	Let $E$ be an $F^+_\alpha$-set. If $h$ is a dimension function satisfying the relation $h(x)\ge C x^\alpha\sqrt{x} \log^{\theta}(\frac{1}{x})$ for $\theta>\frac{2\alpha+1}{2\alpha}$ then $\HH{h}(E)>0$.
\end{cor}

\subsection{The  \texorpdfstring{$\h\to \h\sqrt{\cdot}$}{h-hsqrt-zero} bound, dimension zero}\label{sec:hsqrt0}

In this section we look at a class of very small Furstenberg sets. We will study, roughly speaking, the extremal case of $F_0$-sets and ask ourselves if inequality \eqref{eq:dim} can be extended to this class. Our approach to the problem, using dimension functions, allows us to tackle the problem about the dimensionality of these sets in some cases. We study the case of $F_\h$-sets associated to one particular choice of $\h$. We will look at the function $\h(x)=\dfrac{1}{\log(\frac{1}{x})}$ as a model of ``zero dimensional" dimension function. Our next theorem will show that in this case inequality \eqref{eq:dim} can indeed be extended.
The trick here will be to replace the dyadic scale on the radii in $J_k$ with a faster decreasing sequence $\b=\{b_k\}_{k\in\N}$. 

The main difference will be in the estimate of the quantity of lines in $b_k$-separated directions that intersect two balls of level $J_k$ with a fixed distance $s_k$ between them. This estimate is given by Lemma \ref{lem:conobolas}. Note that the problem in the above bound is the rapid decay of $\h^{-1}$, which is solved by the positivity assumption. In this case, since we are dealing with a zero dimensional function $\h$, the inverse involved decays dramatically to zero. Therefore the strategy cannot be the same as before, where we choose optimally the sequence $\a$. In this case, we will obtain a result by choosing an appropriate sequence of scales. 

\begin{thm}\label{thm:htohsqrt0}
	Let $\h(x)=\frac{1}{\log(\frac{1}{x})}$ and let $E$ be an $F_\h$-set. Then $\dim_H(E)\ge \frac{1}{2}$.
\end{thm}

\begin{proof}
Take a non-negative sequence $\b$ which will be determined later. We will apply the splitting Lemma \ref{lem:split} as in the previous section. For this, take $k_0$ as in \eqref{eq:hiplemasplit} associated to the sequence $\a=\{{k^{-2}}\}_{k\in\N}$. Now, for a given generic $\delta$-covering of $E$ with $\delta<\min\{\delta_E,2^{-k_0}\}$, we use  Lemma \ref{lem:part} to obtain a decomposition $\s=\bigcup_{k\ge k_0}\Omega_k$ with
\begin{equation*}
	\Omega_k=\left\{e\in\s: \hh{\h}(\ell_e\cap E_k)\ge ck^{-2}\right\},
\end{equation*}
where $E_k=E\cap \bigcup_{J_k^\b}B_j$, $J_k^\b$ is the partition of the radii  associated to $\b$ and $c>0$ is a suitable constant. The same calculations as in Theorem \ref{thm:htohsqrt} yield
\begin{equation*}
\#J^\b_{k} \gtrsim \left(\frac{\sigma(\Omega_k)}{b_{k-1}}\right)^{1/2}\frac{\h^{-1}(ck^{-2})^{1/2}}{k^2\h(b_{k-1})}\ge
\left(\frac{\sigma(\Omega_k)}{b_{k-1}}\right)^{1/2}\frac{e^{-ck^2}}{k^2}.
\end{equation*}
Now we estimate a sum like \eqref{eq:cotahaushsqrt}. For  $\beta<\frac{1}{2}$ we have
\begin{equation*}
\sum_{j\ge0} r_j^{\beta}\ge \sum_{k\ge0} \sigma(\Omega_k)^{1/2}\frac{b_k^\beta}{b_{k-1}^{\frac{1}{2}}}\frac{e^{-ck^2}}{k^2}\gtrsim \sqrt{\sum_{k\ge0} \sigma(\Omega_k) \frac{b_k^{2\beta}}{b_{k-1}}\frac{1}{e^{ck^2}k^4}}.
\end{equation*}
In the last inequality we use that the terms are all non-negative. The goal now is to take some rapidly decreasing sequence such that the factor $\frac{b_k^{2\beta}}{b_{k-1}}$ beats the factor $k^{-4}e^{-ck^2}$.  Let us take $0<\e<\frac{1-2\beta}{2\beta}$  and consider the hyperdyadic scale $b_k=2^{-(1+\e)^k}$. With this choice, we have
\begin{equation*}
	\frac{b_k^{2\beta}}{b_{k-1}}=2^{(1+\e)^{k-1}-(1+\e)^{k}2\beta}=2^{(1+\e)^{k}(\frac{1}{1+\e}-2\beta)}.
\end{equation*} 
We obtain that
\begin{equation*}
\left(\sum_{j\ge0} r_j^{\beta}\right)^2 \ge  \sum_{k\ge0} \sigma(\Omega_k) \frac{2^{(1+\e)^k(\frac{1}{1+\e}-2\beta)}}{e^{ck^2}k^4}.
\end{equation*}
Finally, since by the positivity of $\frac{1}{1+\e}-2\beta$ the double exponential in the numerator grows faster than the denominator, we obtain that  
\begin{equation*}
\left(\sum_{j\ge0} r_j^{\beta}\right)^2\gtrsim\sum_{k\ge0} \sigma(\Omega_k)\gtrsim 1.
\end{equation*}
\end{proof}

\begin{cor}
	Let $\theta>0$. If $E$ is an  $F_\h$-set with $\h(x)=\frac{1}{\log^\theta(\frac{1}{x})}$ then $\dim_H(E)\ge \frac{1}{2}$.
\end{cor}
This shows that there is a whole class of $F_0$-sets that must be at least $\frac{1}{2}$-dimensional.

We want to remark that, shortly after \cite{mr10} was published, we were noticed indirectly by Tam\'as Keleti and Andr\'as M\'ath\'e that Theorem \ref{thm:htohsqrt0} can actually be improved. The same result holds for the choice of $\h(x)=\frac{1}{\log\log(\frac1{x})}$ if we use a slightly faster hyperdyadic scale, namely $b_{k}=2^{(1+\e)^{k^3}}$. The improved theorem is the following.

\begin{thm}
Let $\h(x)=\frac{1}{\log\log(\frac{1}{x})}$ and let $E$ be an $F_\h$-set. Then $\dim_H(E)\ge \frac{1}{2}$.
\end{thm}

But this is as far as we can go. They have found, for any $h\prec\h$, an explicit construction of a set $E\in F_{h}$ such that $\dim_{H}(E)=0$.

\section{Fractal sets of directions}\label{sec:fractal}

In this section we will apply our techniques to a more general problem. Consider now the class of Furstenberg sets but defined by a fractal set of directions. Precisely, we have the following definition.

\begin{defn}\label{def:alphabeta}
For  $\alpha,\beta$ in $(0,1]$, a subset $E$ of $\RR$ will be called an $F_{\alpha\beta}$-set if there is a subset $L$ of the unit circle such that $\dim_H(L)\ge\beta$ and, for each direction $e$ in $L$, there is a line segment $\ell_e$ in the direction of $e$ such that the Hausdorff dimension of the set $E\cap\ell_e$ is equal or greater than $\alpha$. 
\end{defn}
This generalizes the classical definition of Furstenberg sets, when the whole circle is considered as set of directions. The purpose here is to study how the parameter $\beta$ affects the bounds above. 
From our results we will derive the following proposition.
\begin{prop}\label{pro:fab}
For any set $E\in F_{\alpha\beta}$, we have that 
\begin{equation}\label{eq:dim_ab}
\dim_H(E)\ge \max\left\{2\alpha+\beta -1 ; \frac{\beta}{2}+\alpha\right\},\qquad \alpha,\beta>0.
\end{equation} 
\end{prop}
It is not hard to prove Proposition \ref{pro:fab} directly, but we will study this problem in a wider scenario and derive it as a corollary. Moreover, by using general Hausdorff measures, we will extend inequalities \eqref{eq:dim_ab} to the zero dimensional case.

\begin{defn}\label{def:hg}
Let $\h$ and $\g$ be two dimension functions. A set $E\subseteq\RR$ is a Furstenberg set of type $\h\g$, or an $F_{\h\g}$-set, if there is a subset $L$ of the unit circle such that $\HH{\g}(L)>0$ and, for each direction $e$ in $L$,  there is a line segment $\ell_e$ in the direction of $e$ such that  $\HH{\h}(\ell_e \cap E)>0$.
\end{defn}
Note that this is the natural generalization of the $F^+_{\alpha\beta}$ class:
\begin{defn}\label{def:alphabeta+}
For each pair $\alpha,\beta\in (0,1]$, a set $E\sub\RR$ will be called an $F^+_{\alpha\beta}$-set if there is a subset $L$ of the unit circle such that $\HH{\beta}(L)>0$ and, for each direction $e$ in $L$,  there is a line segment $\ell_e$ in the direction of $e$ such that  $\HH{\alpha}(\ell_e \cap E)>0$.
\end{defn}
Following the intuition suggested by Proposition \ref{pro:fab}, one could conjecture that if $E$ belong to the class $F_{\h\g}$ then  an appropriate dimension function for $E$ should be dimensionally greater than $\frac{\h^2\g}{\id}$ and $\h\sqrt{\g}$. This will be the case, indeed, and we will provide some estimates on the gap between those conjectured dimension functions and a generic test function $h\in\H$ to ensure that $\HH{h}(E)>0$, and also illustrate with some examples.
We will consider the two results separately. Namely, for a given pair of dimension functions $\h,\g\in\H$, in Section \ref{sec:h2g/id} we obtain sufficient conditions on a test dimension function $h\in\H$, $h\prec \frac{\h^2\g}{\id}$ to ensure that $\HH{h}(E)>0$ for any set $E\in F_{\h\g}$. In Section \ref{sec:hsqrtg} we consider the analogous problem for $h\prec \h\sqrt{\g}$. It turns out that one relevant feature of the set of directions is related to the notion of $\delta$-entropy:

\begin{defn}
 Let $E\subset \RN$ and $\delta\in\R_{>0}$. The $\delta$-entropy  of $E$ is the maximal possible cardinality of a $\delta$-separated subset of $E$. We will denote this quantity with $\NN(E)$.
\end{defn}

The main idea is to relate the $\delta$-entropy to some notion of size of the set. Clearly, the entropy is essentially the Box dimension or the Packing dimension of a set (see \cite{mat95} or \cite{fal03} for the definitions) since both concepts are defined in terms of separated $\delta$ balls with centres in the set. However, for our proof we will need to relate the entropy of a set to some quantity that has the property of being (in some sense) stable under countable unions. One choice is therefore the notion of Hausdorff content, which enjoys the needed properties: it is an outer measure, it is finite, and it  reflects the entropy of a set in the following manner. Recall that the $\g$-dimensional Hausdorff content of a set $E$ is defined as 
\begin{equation*}
\HH{\g}_\infty(E)=\inf\left\{\sum_i\g(\diam(U_i): E\subset \bigcup_i U_i\right\}.
\end{equation*} 

Note that the $\g$-dimensional Hausdorff content $\HH{\g}_\infty$ is clearly not the same than the Hausdorff measure $\HH{\g}$. In fact, they are the measures obtained by applying Method I and Method II (see \cite{mat95}) respectively to the premeasure that assigns to a set $A$ the value $\g(\diam(A))$. For future reference, we state the following estimate for the $\delta$-entropy of a set with positive $\g$-dimensional Hausdorff content as a lemma.
\begin{lem}\label{lem:entropy}
Let $\g\in\H$ and let $A$ be any set. Let  $\NN(A)$ be the $\delta$-entropy of $A$. Then $\NN(A)\ge\frac{\HH{\g}_\infty(A)}{\g(\delta)}$.
\end{lem}
Of course, this result is meaningful when $\HH{\g}_\infty(A)>0$. We will use it in the case in which $\HH{\g}(A)>0$, which is equivalent to $\HH{\g}_\infty(A)>0$. Note that the lemma above only requires the finiteness and the subadditivity of the Hausdorff content. The relevant feature that will be needed in our proof is the $\sigma$-subadditivity, which is a property that the Box dimension does not share. Following the notation of  Definition \ref{def:bk} we have the following analogue of Lemma \ref{lem:part}:
\begin{lem}\label{lem:part_alphabeta}
	Let $E$ be an $F_{\h\g}$-set for some $\h,\g\in\mathbb{H}$ with the directions in $L\subset \s$ and let  $\a=\{a_k\}_{k\in\N}\in\ell^1$ be a non-negative sequence. Let $\mathcal{B}=\{B_j\}$ be a $\delta$-covering of $E$ with $\delta< \delta_E$ and let $E_k$ and $J_k$ be as above. Define
\begin{equation}
	L_k:=\left\{e\in\s: \hh{\h}(\ell_e\cap E_k)\ge \frac{a_k}{2\|\a\|_1}\right\}.\nonumber
\end{equation}
Then $L=\cup_k L_k$.
\end{lem}

\subsection{The Kakeya type bound }\label{sec:h2g/id}

Now we will  prove a generalized version of the bound $\dim_H(E)\ge 2\alpha+\beta-1$ for $E\in F_{\alpha\beta}$. We have the following theorem.
\begin{thm}[$\h\g\to \frac{\h^2\g}{\id}$]\label{thm:hgtoh2g/x}
	Let $\g\in\H$ and $\h\in\mathbb{H}_d$ be two dimension functions and let $E$ be an $F_{\h\g}$-set. Let $h\in\mathbb{H}$  such that $h\prec \frac{\h^2\g}{\emph{\id}}$. For $\delta>0$, let $C(\delta)$ be as in \eqref{eq:maximalgeneral}. If ${\D\sum_{k\ge0}} \sqrt{\frac{\h^2(2^{-k})C(2^{-k+1})}{h(2^{-k})}}<\infty$, then $\HH{h}(E)>0$.
\end{thm}

\begin{proof}
Let $E\in F_{\h\g}$ and let $\{B_j\}_{j\in\N}$ be a covering of $E$ by balls with $B_j=B(x_j;r_j)$. Define $\a=\{a_k\}$ by $a^2_k=\frac{\h^2(2^{-k})C(2^{-k+1})}{h(2^{-k})}$. Therefore, by hypothesis $\a\in\ell^1$. Also define, as in the previous section, for each $k\in\N$, $J_k=\{j\in \N: 2^{-k}< r_j\le 2^{-k+1}\}$ and $E_k=E\cap\D\cup_{j\in J_k}B_j$.
Since $\a\in\ell^1$, we can apply Lemma \ref{lem:part} to obtain the decomposition of the set of directions as $L=\bigcup_k L_k$ associated to this choice of $\a$. We proceed as in the $F_{\alpha}$-class and apply the maximal function inequality to a weighted union of indicator functions:
\[
f:=\h(2^{-k})2^k\D\chi_{F_k}.
\]
As before, 
\begin{equation}\label{eq:L2above_alphabeta}
	\|f\|_{2}^{2}\lesssim \#J_k \h^2(2^{-k}).
\end{equation} 
The same arguments used in the proof of Theorem \ref{thm:htoh2} in Section \ref{sec:lower} allows us to obtain a lower bound for the maximal function. Essentially, the maximal function is pointwise bounded from below by the average of $f$ over the tube centred on the line segment $\ell_e$ for any $e\in L_k$. Therefore, we have the following bound for the $(L^2,\mu)$ norm. Here, $\mu$ is a measure supported on $L$ that obeys the law $\mu(B(x,r)\le\g(r)$ for any ball $B(x,r)$ given by Frostman's lemma.
\begin{equation}\label{eq:maximalbelow_alphabeta}
	\|\mathcal{K}_{2^{-k+1}}(f)\|_{L^2(d\mu)}^2\gtrsim a_k^2\mu(L_k)=\frac{\mu(L_k)\h^2(2^{-k})C(2^{-k})}{h(2^{-k})}.
\end{equation} 
Combining \eqref{eq:maximalbelow_alphabeta} with the maximal inequality \eqref{eq:maximalgeneral}, we obtain
\begin{equation*}
	\frac{\mu(L_k)\h^2(2^{-k})C(2^{-k})}{h(2^{-k})}\lesssim \|\mathcal{K}_{2^{-k+1}}(f)\|_2^2\lesssim C(2^{-k+1})\|f\|_2^2\le C(2^{-k})\|f\|_2^2.
\end{equation*} 
We also have the bound \eqref{eq:L2above_alphabeta}, which implies that  $\frac{\mu(L_k)}{h(2^{-k})}\lesssim \#J_k$, which easily yields the desired result.
\end{proof}
 
\begin{cor}\label{cor:coroh2g/x}
	Let $E$ an $F^+_{\alpha\beta}$-set. If $h$ is any dimension function satisfying $ h(x)\ge C x^{2\alpha+\beta-1} \log^{\theta}(\frac{1}{x})$
for $\theta>2$, then $\HH{h}(E)>0$.
\end{cor}

\begin{rem}
Note that the bound $\dim(E)\ge 2\alpha+\beta-1$ for $E\in F_{\alpha\beta}$ follows directly from this last corollary.
\end{rem}

\subsection{The combinatorial bound}\label{sec:hsqrtg}

In this section we deal with the bound $\h\g\to\h\sqrt{\g}$, which is the significant bound near the endpoint $\alpha=\beta=0$ and generalizes the bound $\dim_H(E)\ge \frac{\beta}{2}+\alpha$ for $E\in F_{\alpha\beta}$. Note that the second bound in \eqref{eq:dim_ab} is meaningless for small values of $\alpha$ and $\beta$. We will again consider separately the cases of $\h$ being zero dimensional or positive dimensional. In the next theorem, the additional condition on $\h$ reflects the positivity of the dimension function. We will use again the two relevant lemmas from Section  \ref{sec:lower}. Lemma \ref{lem:split} is the ``splitting lemma'' and  Lemma \ref{lem:conobolas} is the combinatorial ingredient in the proof of both Theorem \ref{thm:hgtohsqrtg} and Theorem \ref{thm:hgtohsqrtg,hzero,gpos}. We have the following theorem. Recall that $h_\alpha(x)=x^\alpha$.

\begin{thm}[$\h\g\to\h\sqrt{\g}$, $\h\succ h_\alpha$]\label{thm:hgtohsqrtg}
	Let $\g\in\H$, $\h\in\mathbb{H}_d$ be two dimension functions such that $\h(x)\lesssim x^\alpha$ for some $0<\alpha<1$ and let $E$ be an $F_{\h\g}$-set. Let $h\in\mathbb{H}$ with $h\prec \h\sqrt{\g}$.  If ${\D\sum_{k\ge0}} \left(\frac{\h(2^{-k})\sqrt{\g}(2^{-k})}{h(2^{-k})}\right)^{\frac{2\alpha}{2\alpha+1}}<\infty$, then $\HH{h}(E)>0$.
\end{thm}

\begin{proof}
Let $E\in F_{\h\g}$ and let $\{B_j\}_{j\in\N}$ be a covering of $E$ by balls with $B_j=B(x_j;r_j)$. Consider the sequence $\a$ defined as $\a=\left\{\left(\frac{\h\sqrt{\g}}{h}(2^{-k})\right)^{\frac{2\alpha}{2\alpha+1}}\right\}_{k\ge 1}$. Also define, as in the previous section, for each $k\in\N$, $J_k=\{j\in \N: 2^{-k}< r_j\le 2^{-k+1}\}$ and $E_k=E\cap\D\cup_{j\in J_k}B_j$.
Since by hypothesis $\a\in\ell^1$, we can apply Lemma \ref{lem:part} to obtain the decomposition of the set of directions as $L=\bigcup_k L_k$ associated to this choice of $\a$, where $L_k$ is defined as 

\begin{equation}
	L_k:=\left\{e\in\s: \hh{\h}(\ell_e\cap E_k)\ge \frac{a_k}{2\|\a\|_1}\right\}.\nonumber
\end{equation}

Now, let $\{ e^k_{j}\}_{j=1}^{N_k}$ be a $2^{-k}$-separated subset of $L_k$. Taking into account the estimate for the entropy given in Lemma \ref{lem:entropy}. We obtain then that
\begin{equation}\label{eq:entropy_Hinfty}
 N_k\gtrsim\frac{\HH{\g}_\infty(L_k)}{\g(2^{-k})}.
\end{equation} 
We can proceed as in the previous section to obtain that

\begin{equation*}
\#J_{k} \gtrsim a_k\h^{-1}(ra_k)^{1/2}\frac{N_k^{1/2}}{\h(2^{-k})}.
\end{equation*}
Therefore, for any $h\in\H$ as in the hypothesis of the theorem, we have the estimate 
\begin{eqnarray*}
\sum_{j\ge0} h(r_j)& \gtrsim &\sum_{k\ge0} h(2^{-k})\#J_k\\
& \gtrsim &\sum_{k\ge0}a_k\h^{-1}(ra_k)^{1/2} N_k^{1/2}\frac{\sqrt{\g}(2^{-k})}{(\frac{\h\sqrt{\g}}{h})(2^{-k})}.
\end{eqnarray*}
Recall now  that from \eqref{eq:entropy_Hinfty} we have $\sqrt{\g}(2^{-k})N_k^{\frac{1}{2}}\gtrsim \HH{\g}_\infty(L_k)^\frac{1}{2}$. We obtain 
\begin{equation*}
\sum_{j\ge0} h(r_j) \gtrsim  \sum_{k\ge0}\frac{\HH{\g}_\infty(L_k)^{1/2}a_k^{\frac{2\alpha+1}{2\alpha}}}{(\frac{\h\sqrt{\g}}{h})(2^{-k})}
= \sum_{k\ge0}\HH{\g}_\infty(L_k)^{1/2}\gtrsim 1.
\end{equation*}
We use again that $\h(x)\lesssim x^\alpha$ implies that $\h^{-1}(x)\gtrsim x^{\frac{1}{\alpha}}$. In the last inequality, we used the $\sigma$-subadditivity of $\HH{\g}_\infty$.
\end{proof}

\begin{cor}\label{cor:corohsqrtg}
	Let $E$ be an $F^+_{\alpha\beta}$-set for $\alpha,\beta>0$. If $h$ is a dimension function satisfying $h(x)\ge C x^{\frac{\beta}{2}+\alpha} \log^{\theta}(\frac{1}{x})$ for $\theta>\frac{2\alpha+1}{2\alpha}$, then $\HH{h}(E)>0$.

\end{cor}

\begin{rem}
Note that again the bound $\dim(E)\ge \alpha+\frac{\beta}{2}$ for $E\in F_{\alpha\beta}$ follows directly from this last corollary.
\end{rem}

In the next theorem we consider the case of a family of very small Furstenberg sets. More precisely, we deal with a family that corresponds to the case $\alpha=0$, $\beta\in (0,1]$ in the classical setting.

\begin{thm}[$\h\g\to\h\sqrt{\g}$, $\h$ zero dimensional, $\g$ positive]\label{thm:hgtohsqrtg,hzero,gpos}
	Let $\beta>0$ and define $\g(x)=x^\beta, \h(x)=\frac{1}{\log\log(\frac{1}{x})}$.  If $E$ is an $F_{\h\g}$-set, then $\dim(E)\ge \frac{\beta}{2}$.
\end{thm}
The proof follows from the same ideas as in Theorem Theorem \ref{thm:htohsqrt0} in Section \ref{sec:lower}, with the natural modifications. We have the following immediate corollary.

\begin{cor}
	Let $\theta>0$. If $E$ is an  $F_{\h\g}$-set with $\h(x)=\frac{1}{\log\log^\theta(\frac{1}{x})}$ and $\g(x)=x^\beta$, then $\dim(E)\ge \frac{\beta}{2}$.
\end{cor}

The next question would be: Which should it be the expected dimension function for an $F_{\h\g}$-set if 
$\h(x)=\g(x)=\frac{1}{\log(\frac{1}{x})}$? The preceding results lead us to the following conjecture:
\begin{conj}
Let  $\h(x)=\g(x)=\frac{1}{\log(\frac{1}{x})}$ and let $E$ be an $F_{\h\g}$-set. Then  $\frac{1}{\log^\frac{3}{2}(\frac{1}{x})}$ should be an appropriate dimension function for $E$, in the sense that a logarithmic gap can be estimated.
\end{conj}
We do not know, however,  how to prove this.

\subsection{A remark on the notion of size for the set of directions}\label{sec:size}

We have emphasized that the relevant ingredient for the combinatorial proof in Section \ref{sec:hsqrtg} is the notion of $\delta$-entropy of a set. In addition, we have discussed the possibility of consider the Box dimension as an adequate notion of size to detect this quantity. In this section we present an example that shows that in fact the notion of Packing measure is also inappropriate.
We want to remark here that none of them will give any further (useful) information to this problem and therefore there is no chance to obtain similar results in terms of those notions of dimensions. To make it clear, consider the classical problem of proving the bound $\dim_H(E)\ge\alpha+\frac{\beta}{2}$ for any $E\in F_{\alpha\beta}$ where $\beta$ is the Box or Packing dimension of the set $L$ of directions. 

We illustrate this remark with the extreme case of $\beta=1$. It is absolutely trivial that nothing meaningful can be said if we only know that the Box dimension ($\dim_B$) of $L$ is $1$, since any countable dense subset $L$ of $\s$ satisfies $\dim_B(L)=1$ but in that case, since $L$ is countable, we can only obtain that $\dim_H(E)\ge\alpha$. 

For the Packing dimension ($\dim_P$), it is also easy to see that if we only know that $\dim_P(L)=1$ we do not have any further information about the Hausdorff dimension of the set $E$. To see why, consider the following example. Let $C_\alpha$ be a regular Cantor set such that $\dim_H(C_\alpha)=\dim_B(C_\alpha)=\alpha$. Let $L$ be a set of directions with $\dim_H(L)=0$ and $\dim_P(L)=1$. Now, we build the Furstenberg set $E$ in polar coordinates as
\begin{equation*}
 E:=\{(r,\theta): r\in C_\alpha, \theta\in L\}.
\end{equation*} 
This can be seen as a ``Cantor target'', but with a fractal set of directions instead of the whole circle. By the Hausdorff dimension estimates, we know that $\dim_H(E)\ge\alpha$. We show that in this case we also have that $\dim_H(E)\le \alpha$, which implies that in the general case this is the best that one could expect, even with the additional information about the Packing dimension of $L$. For the upper bound, consider the function $f:\RR\to\RR$ defined by $f(x,y)=(x\cos y,x\sin y)$. Clearly $E=f(C_\alpha\times L)$. Therefore, by the known product formulae that can be found, for example, in \cite{fal03}, we have that
\begin{equation*}
 \dim_H(E)=\dim_H(f(C_\alpha\times L)) \le \dim_H(C_\alpha\times L)= \dim_B(C_\alpha)+\dim_H(L)=\alpha.
\end{equation*}

\section{Upper bounds}\label{sec:upper}

In this section we look at a refinement of the upper bound for the dimension of Furstenberg sets. Since we are looking for upper bounds on a class of Furstenberg sets, the aim will be to explicitly construct  a very small set belonging to the given class.

We first consider the classical case of power functions,  $x^\alpha$, for $\alpha >0$. Recall that for this case, the known upper bound implies that, for any positive
$\alpha$, there is a set $E\in F_\alpha$ such that
$\HH{\frac{1+3\alpha}{2}+\e}(E)=0$ for any $\e>0$. By looking closer at  Wolff's arguments, it can be seen that in fact it is true that $\HH{g}(E)=0$ for any dimension function $g$ of the form
\begin{equation}\label{eq:supertrivial}
g(x)=x^\frac{1+3\alpha}{2}\log^{-\theta}\left(\frac{1}{x}\right),\qquad \theta>\frac{3(1+3\alpha)}{2}+1.
\end{equation}
Further, that argument can be modified (Theorem \ref{thm:logFalpha}) to sharpen on the logarithmic gap, and therefore improving \eqref{eq:supertrivial} by proving the same result for any $g$ of the form
\begin{equation}\label{eq:mediumtrivial}
g(x)=x^\frac{1+3\alpha}{2}\log^{-\theta}\left(\frac{1}{x}\right),\qquad \theta>\frac{1+3\alpha}{2}.
\end{equation}

However, this modification will not be sufficient for our main objective, which is to reach the zero dimensional case. More precisely, we will focus at the endpoint $\alpha=0$, and give a
complete answer about the \emph{exact dimension} of  a class of Furstenberg sets. We will prove in Theorem \ref{thm:sqrth3/2} that, for any given $\gamma>0$, there exists a set
$E_\gamma\sub\R^2$ such that 
\begin{equation}\label{eq:smallmain}
 E_\gamma\in F_{\h_\gamma} \text{ for }\h_\gamma(x)=\frac{1}{\log^\gamma(\frac{1}{x})} \text{ and } \dim_H(E_\gamma)\le\frac{1}{2}.
\end{equation}
This result, together with the results from \cite{mr10} mentioned above, shows that $\frac{1}{2}$ is sharp for the class $F_{\h_\gamma}$. In fact, for this family both  inequalities in \eqref{eq:dim} are in fact the equality $\Phi(F_{\h_\gamma})=\frac{1}{2}$.

In order to be able to obtain \eqref{eq:smallmain}, it is not enough to simply ``refine'' the construction of Wolff. He achieves the desired set by choosing a  specific set as the fiber in each direction. This set is known to have the correct dimension. To be able to reach the zero dimensional case, we need to handle the delicate issue of choosing an analogue zero dimensional set on each fiber. The main difficulty lies in being able to handle \emph{simultaneously} Wolff's construction and the proof of the fact that the fiber satisfies the stronger condition of having positive measure for the correct dimension function.

We will also focus at the endpoint $\alpha=0$, and give a complete answer about the size of a class of Furstenberg sets by proving that (Theorem \ref{thm:sqrth3/2}),  for any given $\gamma>0$, there exists a set $E_\gamma\sub\R^2$ such that $E_\gamma\in F_{\h_\gamma}$ for $\h_\gamma(x)=\frac{1}{\log^\gamma(\frac{1}{x})}$ and $\dim_H(E_\gamma)\le\frac{1}{2}$.

\subsection{Upper Bounds for classical Furstenberg-type Sets}\label{sec:furs}
We begin with a preliminary lemma about a very well distributed (mod 1) sequence.
\begin{lem}\label{lem:discre}
 For $n\in \N$ and any real number $x\in [0,1]$, there is a pair $0\le j,k\le n-1$ such that
\begin{equation*}
\left|x-\left(\sqrt{2}\frac{k}{n}-\frac{j}{n}\right)\right|\le\frac{\log(n)}{n^2}.
\end{equation*}
\end{lem}
This lemma is a consequence of Theorem 3.4 of \cite{kn74}, p125, in which an estimate is given about the discrepancy of the
fractional part of the sequence $\{n\alpha\}_{n\in\N}$ where $\alpha$ is a irrational of a certain type. We also need to introduce the notion of $G$-sets, a common ingredient in the construction of Kakeya and Furstenberg sets.

\begin{defn}
 A $G$-set is a compact set $E\sub\R^2$ which is contained in the strip $\{(x,y)\in\R^2:0\le x \le 1\}$ such that for any $m\in[0,1]$ there is a line segment contained in $E$ connecting $x=0$ with $x=1$ of slope $m$.
\end{defn}
Given a line segment $\ell(x)=mx+b$, we define the $\delta$-tube associated to $\ell$ as
\[
S_\ell^\delta:=\{(x,y)\in\R^2:0\le x\le 1; |y-(mx+b)|\le\delta\}.
\]
\begin{thm}\label{thm:logFalpha}
For $\alpha\in(0,1]$ and $\theta>0$, define
$h_\theta(x)=x^{\frac{1+3\alpha}{2}}\log^{-\theta}(\frac{1}{x})$.
Then, if $\theta>\frac{1+3\alpha}{2}$, there exists a set
$E\in F_\alpha$ with $\HH{h_\theta}(E)=0$.
\end{thm}

\begin{proof}
 
Fix $n\in\N$ and let $n_j$  be a sequence such that $n_{j+1}>n_j^j$. We consider $T$ to be the set defined as follows:
\begin{equation*}
T=\left\{x\in\left[\frac{1}{4},\frac{3}{4}\right]: \forall j \
\exists\ p,q\ ; q\le n_j^\alpha;
\left|x-\frac{p}{q}\right|<\frac{1}{n_j^2}\right\}.
\end{equation*}
It can be seen that $\dim_H(T)=\alpha$. This is a version of Jarn\'ik's theorem on Diophantine Approximation (see \cite{wol99b}, p. 10 and \cite{fal86}, p. 134, Theorem 8.16(b)). If $\varphi(t)=\frac{1-t}{t\sqrt{2}}$ and $D=\varphi^{-1}\left([\frac{1}{4},\frac{3}{4}]\right)$, we have that $\varphi:D\to [\frac{1}{4},\frac{3}{4}]$ is bi-Lipschitz. Therefore the set
\begin{equation*}
T'=\left\{t\in\R:\frac{1-t}{t\sqrt{2}}\in
T\right\}=\varphi^{-1}(T)
\end{equation*}
also has Hausdorff dimension $\alpha$. The main idea of our proof, is to construct a set for which we have, essentially, a copy of $T'$ in each direction and simultaneously keep some optimal covering property. Define, for each $n\in \N$,
\begin{equation*}
\Gamma_n:=\left\{\frac{p}{q}\in\left[\frac{1}{4},\frac{3}{4}\right],
q\le n^\alpha\right\}
\end{equation*}
and
\begin{equation*}
Q_n=\left\{ t:\frac{1-t}{\sqrt{2}t}=\frac{p}{q}\in
\Gamma_n\right\}=\varphi^{-1}(\Gamma_n).
\end{equation*}

To count the elements of $\Gamma_n$ (and $Q_n$), we take into account that
\begin{equation*}
\sum_{j=1}^{\lfloor n^{\alpha}\rfloor}j\le \frac{1}{2} \lfloor
n^{\alpha}\rfloor(\lfloor n^{\alpha}\rfloor+1) \lesssim \lfloor
n^{\alpha}\rfloor^2\le n^{2\alpha}.
\end{equation*}
Therefore, $\#(Q_n)\lesssim n^{2\alpha}$. For $0\le j,k\le n-1$, define the line segments
\[
\ell_{jk}(x):=(1-x)\frac{j}{n}+x\sqrt{2}\frac{k}{n} \text{ for }
x\in[0,1],
\]
and their $\delta_n$-tubes $S_{\ell_{jk}}^{\delta_n}$ with $\delta_n=\frac{\log(n)}{n^2}$. We will use during the proof the notation $S_{jk}^n$ instead of $S_{\ell_{jk}}^{\delta_n}$. Also define 
\begin{equation}\label{eq:Gn}
G_n:=\bigcup_{jk}S_{jk}^n.
\end{equation}
Note that, by Lemma \ref{lem:discre}, all the $G_n$ are $G$-sets.  For each $t\in Q_n$, we look at the points $\ell_{jk}(t)$, and 
define the set $S(t):=\{\ell_{jk}(t)\}_{j,k=1}^n$. Clearly, $\#(S(t))\le n^2$. But if we note that, if $t\in Q_n$, then
\begin{equation*}
0\le\frac{\ell_{jk}(t)}{t\sqrt{2}}=\frac{1-t}{t\sqrt{2}}\frac{j}{n}+\frac{k}{n}=\frac{p}{q}\frac{j}{n}+\frac{k}{n}=\frac{pj+kq}{nq}<2,
\end{equation*}
we can bound $\#(S(t))$ by the number of non-negative rationals smaller than 2 of denominator $qn$. Since $q\le n^\alpha$, we
have $\#(S(t))\le n^{1+\alpha}$. Considering \emph{all} the elements of $Q_n$, we obtain $\#\left(\bigcup_{t\in
Q_n}S(t)\right)\lesssim n^{1+3\alpha}$. Let us define
\begin{equation}\label{eq:lambda_n}
\Lambda_n:=\left\{(x,y)\in G_n: |x-t|\le \frac{\sqrt{2}}{n^2}\text{ for
some } t\in Q_n\right\}.
\end{equation}

\begin{cla}
 For each $n$, take $\delta_n=\frac{\log(n)}{n^2}$. Then $\Lambda_n$ can be covered by $L_n$ balls of radio $\delta_n$ with $L_n \lesssim n^{1+3\alpha}$.
\end{cla}
To see this, it suffices to set a parallelogram on each point of
$S(t)$ for each $t$ in $Q_n$. The lengths of the sides of the
parallelogram are of order $n^{-2}$ and $\frac{\log(n)}{n^2}$, so
their diameter is bounded by  a constant times
$\frac{\log(n)}{n^2}$, which proves the claim.

We   can now begin with the recursive construction that leads to the
desired set. Let $F_0$ be a $G$-set written as
\begin{equation*}
F_0=\bigcup_{i=1}^{M_0}S_{\ell_i^0}^{\delta^0},
\end{equation*}
(the union of $M_0$ $\delta^0$-thickened line segments
$\ell^0_i=m^0_i+b^0_i$ with appropriate orientation). Each $F_j$
to be constructed will be a $G$-set of the form
\begin{equation*}
F_j:=\bigcup_{i=1}^{M_j}S_{\ell_i^j}^{\delta^j}, \qquad \text{
with } \ \ell^j_i=m^j_i+b^j_i.
\end{equation*}
 Having constructed $F_j$, consider the $M_j$ affine mappings
\begin{equation*}
A^j_i:[0,1]\times[-1,1]\rightarrow
S_{\ell^j_{i}}^{\delta^j}\qquad 1\le i\le M_j,
\end{equation*}
defined by
\begin{equation*}
  A^j_i\left(\begin{array}{c}
    x\\
    y
\end{array}\right)=\left(\begin{array}{cc}
    1 & 0\\
    m_i^j & \delta^j
\end{array}\right)\left(\begin{array}{c}
    x\\
    y
\end{array}\right)+\left(\begin{array}{c}
    0\\
    b^j_{i}
\end{array}\right).
\end{equation*}
Here is the key step: by the definition of  $T$, we can choose the sequence $n_{j}$ to grow as  fast  as we need  (this will not be the case in the next section). For example, we can choose $n_{j+1}$ large enough to satisfy
\begin{equation}\label{eq:njloglogMj}
\log\log(n_{j+1})>M_j
\end{equation}
and apply $A_i^j$ to the sets $G_{n_{j+1}}$ defined in
\eqref{eq:Gn} to obtain
\begin{equation*}
F_{j+1}=\bigcup_{i=1}^{M_j}A_i^j(G_{n_{j+1}}).
\end{equation*}
Since $G_{n_{j+1}}$ is a union of thickened line segments, we
have that
\begin{equation*}
F_{j+1}=\bigcup_{i=1}^{M_{j+1}}S_{\ell_i^{j+1}}^{\delta^{j+1}},
\end{equation*}
for an appropriate choice of $M_{j+1}$, $\delta_{j+1}$ and
$M_{j+1}$ line segments $\ell_i^{j+1}$. From the definition of
the mappings $A_i^j$ and since the set $G_{n_{j+1}}$ is a
$G$-set, we conclude that $F_{j+1}$ is also a $G$-set. Define
\begin{equation*}
E_j:=\{(x,y)\in F_j: x\in T'\}.
\end{equation*}
To cover $E_j$, we note that if $(x,y)\in E_j$, then $x\in T'$,
and therefore there exists a rational $\frac{p}{q}\in
\Gamma_{n_j}$ with
\begin{equation*}
\frac{1}{n_j^2}>\left|\frac{1-x}{x\sqrt{2}}-\frac{p}{q}\right|=|\varphi(x)-\varphi(r)|\ge\frac{|x-r|}{\sqrt{2}},\qquad
\text{ for some }\ r\in Q_{n_j}.
\end{equation*}
Therefore $(x,y)\in
\bigcup_{i=1}^{M_{j-1}}A^{j-1}_i(\Lambda_{n_j})$, so we conclude
that $E_j$ can be covered by $M_{j-1}n_{j}^{1+3\alpha}$ balls of
diameter at most $\frac{\log(n_{j})}{n_{j}^2}$. Since we chose the number
$n_j$ such that $\log\log(n_j)>M_{j-1}$, we obtain that $E_{j}$
admits a covering by $\log\log(n_{j})n_{j}^{1+3\alpha}$ balls of
the same diameter. Therefore, if we set $F=\bigcap_j F_j$ and
$E:=\{(x,y)\in F: x\in T'\} $ we obtain that
\begin{eqnarray*}
\HH{h_\theta}_{\delta_j}(E) & \lesssim  & n^{1+3\alpha}_{j}\log\log(n_j)h_\theta\left(\frac{\log(n_j)}{n^2_j}\right)\\
    & \lesssim & n^{1+3\alpha}_{j}\log\log(n_j)\left(\frac{\log(n_j)}{n^2_j}\right)^{\frac{1+3\alpha}{2}}\log^{-\theta}\left(\frac{n^2_j}{\log(n_j)}\right)\\
    & \lesssim & \log\log(n_j)\log(n_j)^{\frac{1+3\alpha}{2}-\theta}\lesssim  \log^{\frac{1+3\alpha}{2}+\e-\theta}(n_j)
\end{eqnarray*}
for large enough $j$. Therefore, for any
$\theta>\frac{1+3\alpha}{2}$, the last expression goes to zero.
In addition, $F$ is a $G$-set, so it must contain a line segment
in each direction $m\in[0,1]$. If $\ell$ is such a line segment,
then
\begin{equation*}
\dim_H(\ell\cap E)=\dim_H(T')\ge\alpha.
\end{equation*}
The final set of the proposition is obtained by taking eight
copies of $E$, rotated to achieve \emph{all} the directions in
$\s$.
\end{proof}

\subsection{Upper Bounds for very small Furstenberg-type Sets}\label{sec:gfurs}

In this section we will focus on the class $F_\alpha$ at the
endpoint $\alpha=0$. Note that all preceding results involved only the case for which $\alpha > 0$. Introducing the generalized Hausdorff measures, we are able to handle an important class of Furstenberg type sets in $F_0$. 

The idea is to follow the proof of Theorem \ref{thm:logFalpha}. But in order to do that, we need to replace the set $T$ by a generalized version of it. A na\"ive approach would be to replace the $\alpha$ power in the definition of $T$ by a slower increasing function, like a logarithm. But in this case it is not clear that the set $T$ fulfills the condition of having positive measure for the corresponding dimension function (recall that we want to construct a set in $F_{\h_\gamma}$).  More precisely, we will need the following lemma.

\begin{lem}\label{lem:fiber}
Let $r>1$ and consider the sequence $\n=\{n_j\}$ defined by $n_j=e^{\frac{1}{2}n_{j-1}^{\frac{4}{r}j}}$, the function 
$\f(x)=\log(x^2)^\frac{r}{2}$ and the set
\begin{equation*}
T=\left\{x\in\left[\frac{1}{4},\frac{3}{4}\right]\setminus \Q: \forall j \ \exists\ p,q\ ; q\le
\f(n_j); |x-\frac{p}{q}|<\frac{1}{n_j^2}\right\}.
\end{equation*}
Then we have that $\HH{\h}(T)>0$ for $\h(x)=\frac{1}{\log(\frac{1}{x})}$.
\end{lem}
This is the essential lemma for our construction. It is trivial that $T$ is a set of Hausdorff dimension zero, but in order to use this set in each fiber of an $F_\h$-set, we have to prove that $T$ has positive $\HH{\h}$-mass. This is the really difficult part. The proof is a more technical version of the a classical result that can be found in \cite{fal03}. For the proof of our lemma, we refer to \cite{mr12}. Both classical and generalized results are examples of Diophantine Approximation. We emphasize the following fact: in this  case, the construction of this new set $T$, does not allow us, as in \eqref{eq:njloglogMj}, to freely choose the sequence $n_j$. On one  hand we need the sequence to be quickly increasing to prove that the desired set is small enough, but not arbitrarily fast, since on the other hand, we need to impose some control to be able to prove that the fiber has the appropriate \emph{positive} measure. 

With this lemma, we are able to prove the main result of this section. We have the next theorem.

\begin{thm}\label{thm:sqrth3/2}
Let $\h=\frac{1}{\log(\frac{1}{x})}$. There exists a set $E\in
F_\h$ such that $\dim_H(E)\le \frac{1}{2}$.
\end{thm}
\begin{proof}
We will use essentially a copy of $T$ in each direction in the construction of the desired set to fulfill the conditions required to be an $F_\h$-set. Let $T$ be the set defined in Lemma \ref{lem:fiber}. Define $T'=\varphi^{-1}(T)$, where $\varphi$ is the same bi-Lipschitz function from the proof of Theorem \ref{thm:logFalpha}. Then $T'$ has positive $\HH{\h}$-measure. 

Let us define the corresponding sets of Theorem \ref{thm:logFalpha} for this generalized case. For $\f(x)=\f(x)=\log(x^2)^\frac{r}{2}$, define 
\begin{equation*}
\Gamma_n:=\left\{\frac{p}{q}\in\left[\frac{1}{4},\frac{3}{4}\right],
q\le \f(n)\right\}\qquad,\qquad
Q_n=\left\{ t:\frac{1-t}{\sqrt{2}t}=\frac{p}{q}\in
\Gamma_n\right\}=\varphi^{-1}(\Gamma_n).
\end{equation*}
Now the estimate is $\#(Q_n)\lesssim \f^{2}(n)=\log^r(n^2)\sim\log^r(n)$. For each $t\in Q_n$, define $S(t):=\{\ell_{jk}(t)\}_{j,k=1}^n$.
If $t\in Q_n$, following the previous ideas, we obtain that $\#(S(t))\lesssim n\log^\frac{r}{2}(n)$,
and therefore
\begin{equation*}
\#\left(\bigcup_{t\in Q_n}S(t)\right)\lesssim
n\log(n)^{\frac{3r}{2}}.
\end{equation*}
Now we estimate the size of a covering of the set $\Lambda_n$ in \eqref{eq:lambda_n}. For each $n$, take $\delta_n=\frac{\log(n)}{n^2}$. As before, the set $\Lambda_n$ can be covered with  $L_n$ balls of radio $\delta_n$ with $L_n\lesssim n\log(n)^\frac{3r}{2}$.

Once again, define $F_j$, $F$, $E_j$ and $E$ as before. Now the
sets $F_j$ can be covered by less than
$M_{j-1}n_j\log(n_j)^\frac{3r}{2}$ balls of diameter at most
$\frac{\log(n_{j})}{n_{j}^2}$. Now we can verify that, since each
$G_n$ consist of $n^2$ tubes, we have that $M_j=M_0n_1^2\cdots
n_j^2$. We can also verify that the sequence $\{n_j\}$ satisfies
the relation $\log{n_{j+1}}\ge M_{j}=M_0n_1^2\cdots n_j^2$, and
therefore we have the bound

\begin{equation*}
\dim_H(E)\le\underline{\dim}_B(E)\le\varliminf_j\frac{\log\left(\log(n_{j})
n_j\log(n_j)^\frac{3r}{2}
\right)}{\log\left(n_{j}^2\log^{-1}(n_j)\right)}=\frac{1}{2},
\end{equation*}
 where $\underline{\dim}_B$ stands for the lower box dimension.
Finally, for any $m\in[0,1]$ we have a line segment $\ell$ with
slope $m$ contained in $F$. It follows that $\HH{\h}(\ell\cap
E)=\HH{\h}(T')>0$.
\end{proof}

We remark that the argument in this particular result is
essentially the same needed to obtain the family of Furstenberg
sets $E_\gamma\in F_{\h_\gamma}$ for
$\h_\gamma(x)=\frac{1}{\log^\gamma(\frac{1}{x})}$,
$\gamma\in\R_+$, such that $\dim_H(E_\gamma)\le \frac{1}{2}$
announced in the introduction.

\subsection{The case \texorpdfstring{$\alpha=0$, $K$}{alpha0-k} points}
Let us begin with the definition of the class $F^K$.
\begin{defn}\label{def:FK}
For $K\in\N$, $K\ge 2$, a set will be a $F^K$-set or a
Furstenberg set of type $K$ if for any direction $e\in\s$, there
are at least $K$ points contained in $E$ lined up in the
direction of $e$.
\end{defn}

Already in \cite{mr10} we proved that there is a $F^2$-set with zero Hausdorff dimension (see also \cite{fal03}, Example 7.8). We will generalize this example to obtain even smaller $F^2$ sets. Namely, for any $h\in\H_0$, there exists $G$ in $F^2$ such that $\HH{h}(G)=0$.
It is clear that  the set $G$ will depend on the choice of $h$. 

\begin{ex}\label{ex:F2Haus_h0}

Given a function $h\in\H$, we will construct two small sets $E, F\sub[0,1]$ with $\HH{h}(E)=\HH{h}(F)=0$ and such that $[0,1]\sub E+F$. Consider now $G=E\times \{1\}\cup -F\times\{0\}$. Clearly, we have that $\HH{h}(G)=0$, and contains two points in every direction $\theta\in [0;\frac{\pi}{4}]$. For, if
$\theta\in [0;\frac{\pi}{4}]$, let $c=\tan(\theta)$, so $c\in[0,1]$. By the choice of $E$ and  $F$, we can find $x\in E$ and $y\in F$ with $c=x+y$. The points $(-y,0)$ and $(x,1)$ belong to $G$ and determine a segment in the direction $\theta$.

\begin{figure}[ht]
 \begin{tikzpicture}[scale=3]
 \path (1.5,.2) coordinate (O);
 \path (O) ++(0:1) coordinate (X);
 \path (O) ++(90:1) coordinate (Y+);
 \path (O) ++(180:1) coordinate (X-);
 \draw[->] (O) -- (X);
 \draw[->] (O) -- (Y+);
 \draw[->] (O) -- (X-);
 \path (O) ++(0:1) coordinate (X1);
 \path (O) ++(90:1) coordinate (Y1)node[above left]{\footnotesize $1$};
 \path (Y1) ++(0:1) coordinate (XY1);
 \path (Y1) ++(180:1) coordinate (Xmenos1Y1);
 \path (Y1) ++(0:.6) coordinate (maxE) node[below right]{\scriptsize $E\times\{1\}$};
 \draw[dashed] (Xmenos1Y1) -- (XY1);
 \path (Y1) ++(0:.4) coordinate (e)node[above]{\footnotesize $e$};
 \path (O) ++(180:.15) coordinate (mf)node[below]{\footnotesize $-f$};
 \path (O) ++(180:.4) coordinate (maxmF);
 \draw[line width=1.5pt, color=black] (Y1)--(maxE);
 \draw[line width=1.5pt, color=black] (O)--(maxmF)node[below left]{\scriptsize $-F\times\{0\}$};
 \draw[fill=black] (e) circle (.02);
 \draw[fill=black] (mf) circle (.02);
 \draw[line width=1pt, dotted] (mf) -- (Y1);
 \draw[line width=1pt, dotted] (mf) -- (e);
 \draw[line width=1pt, dotted] (mf) -- (maxE);
 \draw[line width=1pt, dashed] (maxmF) -- (Y1);
 \draw[line width=1pt, dashed] (maxmF) -- (e);
 \draw[line width=1pt, dashed] (maxmF) -- (maxE);
 \end{tikzpicture}
\caption{An $F^2$-set of zero $\HH{h}$-measure}
\end{figure}
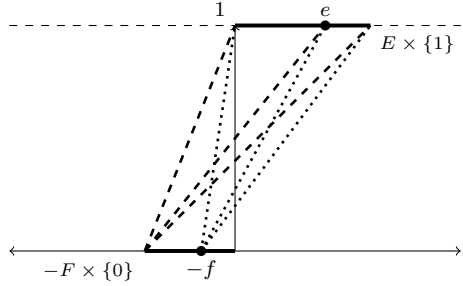

For $x\in [0,1]$, we consider its binary representation $x=\sum_{j\ge 1} r_j 2^{-j}$, $r_j\in\{0,1\}$. We define 
 $E:=\left\{x\in [0,1]: r_j=0 \text{ if } m_k+1\le j\le m_{k+1}; k \text{ even} \right\}$ and $F:=\left\{x\in [0,1]: r_j=0 \text{ if } m_k+1\le j\le m_{k+1}; k \text{ odd } \right\}$. Here $\{m_k; m_0=0\}_k$ is an increasing sequence such that
$m_k\to+\infty$.  Now we estimate the size of the set $E$. Given
$k\in\N$, $k$ even, define $\ell_k=m_k-m_{k-1}+\dots+m_2-m_1$. It
is clear that $E$ can be covered by $2^{l_k}$ intervals of length
$2^{-m_{k+1}}$. Therefore, if the sequence $m_k$ increases fast
enough, then $\dim_H(E)\le\lbd(E)\le\varliminf_k \frac{\log(2^{\ell_k})}{\log (2^{m_{k+1}})}\lesssim\varliminf_k \frac{2^{\ell_k}}{2^{m_{k+1}}}=0$. Since the same argument shows that $\dim_H(F)=0$, this estimate proves that the set $G$ has Hausdorff dimension equal to zero. Now, for the finer estimate on the $\HH{h}$-measure of the set, we must impose a more restrictive condition on the sequence $\{m_k\}$. Recall that the covering property implies that, for a given $h\in\H$, we have that $  \HH{h}(E)\le 2^{l_k}h(2^{-m_{k+1}})$.
Therefore we need to choose a sequence $\{m_j\}$, depending on $h$, such that the above quantity goes to zero with $k$. Since $\ell_k\le m_k$, we can define recursively the sequence $\{m_k\}$ to satisfy the relation $ 2^{m_k}h(2^{-m_{k+1}})= \frac{1}{k}$. This last condition is equivalent to $m_{k+1}=\log\left(\frac{1}{h^{-1}(\frac{1}{k2^{m_k}})}\right)$. As an concrete example, take $h(x)=\frac{1}{\log(\frac{1}{x})}$. In this case we obtain that the sequence $\{m_k\}$ can be defined as $m_{k+1}=k2^{m_k}$.
\end{ex}

\subsection{Remark about the Packing dimension for small Furstenberg sets}
It is worthy to note here that if we were to measure the size of
Furstenberg sets with the packing dimension, the situation is
absolutely different. More precisely, for $K\ge 2$, any $F^K$-set
$E\subset \RR$ must have $\dim_P(E)\ge \frac{1}{2}$. For, if $E$
is an $F^2$ set, then the map $\varphi$ defined by
 $\varphi(a,b)=\frac{a-b}{\|a-b\|}$
is Lipschitz when restricted to $G_\e:=E\times
E\setminus\{(x,y)\in E\times E: \|(x,y)-(a,a)\|<\e; a\in E\}$.
Roughly, we are considering the map that recovers the set
of directions but restricted ``off the diagonal''. It is clear
that we can assume without loss of generality that all the pairs
are the endpoints of unit line segments. Therefore, since $E$ is
an $F^K$-set, $\varphi(G_\e)=\s$ if $\e$ is small enough. We
obtain the inequality
\begin{equation*}
 1=\dim_H(\s)\le \dim_H(G_\e)\le\dim_H(E\times E).
\end{equation*}
The key point is the product formulae for Hausdorff and Packing
dimensions. We obtain that
\begin{equation*}
 1\le\dim_H(E\times E)\le \dim_H(E)+\dim_P(E)\le 2\dim_P(E)
\end{equation*}
 and then $\dim_P(E)\ge\frac{1}{2}$. It also follows that if we achieve small Hausdorff dimension then the Packing dimension is forced to increase. In particular, the $F^2$-set constructed in \cite{mr10} has Hausdorff dimension 0 and therefore it has Packing dimension 1. 
 
The following construction can be understood as optimal in the sense of obtaining the smallest possible dimensions, both Hausdorff and Packing. There is an $F^2$ set $E$ such that $\dim_H(E)=\frac{1}{2}=\dim_P(E)$:
\begin{ex}\label{ex:F2Haus=pack=1/2}
The construction is essentially the same as in Example \ref{ex:F2Haus_h0}, but we use two
different sets to obtain all directions. Let $A$ be the set of
all the numbers whose expansion in base $4$ uses only the digits
$0$ and $1$. On the other hand, let $B$ the set of those numbers
which only uses the digits $0$ and $2$. Both sets have Packing
and Hausdorff dimension equal to $\frac{1}{2}$ and
$[0,1]\subseteq A+B$. The construction follows then the same pattern
as in the previous example.
\end{ex}

\section{Acknowledgements}
This expository article was completed during my stay at the Departamento de An\'alisis Matem\'atico, Universidad de Sevilla. I am deeply grateful in particular to Professor Carlos P\'erez Moreno for his hospitality. I would also like to thank my PhD advisor Ursula Molter from Universidad de Buenos Aires for her guidance and support.

\bibliographystyle{alpha}

\begin{thebibliography}{CMMS04}

\bibitem[Bes19]{bes19}
A.~S. Besicovitch.
\newblock Sur deux questions d'int\'egrabilit\'e des fonctions.
\newblock {\em J. Soc. Phys.-Math. (Perm')}, 2:105--123, 1919.

\bibitem[Bes28]{bes28}
A.~S. Besicovitch.
\newblock On {K}akeya's problem and a similar one.
\newblock {\em Math. Z.}, 27(1):312--320, 1928.

\bibitem[Bes56a]{bes56a}
A.~S. Besicovitch.
\newblock On density of perfect sets.
\newblock {\em J. London Math. Soc.}, 31:48--53, 1956.

\bibitem[Bes56b]{bes56b}
A.~S. Besicovitch.
\newblock On the definition of tangents to sets of infinite linear measure.
\newblock {\em Proc. Cambridge Philos. Soc.}, 52:20--29, 1956.

\bibitem[Bou94]{bou94}
Jean Bourgain.
\newblock Hausdorff dimension and distance sets.
\newblock {\em Israel J. Math.}, 87(1-3):193--201, 1994.

\bibitem[CHM10]{chm10}
Carlos~A. Cabrelli, Kathryn~E. Hare, and Ursula~M. Molter.
\newblock Classifying {C}antor sets by their fractal dimensions.
\newblock {\em Proc. Amer. Math. Soc.}, 138(11):3965--3974, 2010.

\bibitem[CMMS04]{cmms04}
Carlos Cabrelli, Franklin Mendivil, Ursula Molter, and Ronald Shonkwiler.
\newblock On the {H}ausdorff {$h$}-measure of {C}antor sets.
\newblock {\em Pacific J. Math.}, 217(1):45--59, 2004.

\bibitem[Dav71]{dav71}
Roy~O. Davies.
\newblock Some remarks on the {K}akeya problem.
\newblock {\em Proc. Cambridge Philos. Soc.}, 69:417--421, 1971.

\bibitem[Egg52]{egg52}
H.~G. Eggleston.
\newblock Sets of fractional dimensions which occur in some problems of number
  theory.
\newblock {\em Proc. London Math. Soc. (2)}, 54:42--93, 1952.

\bibitem[EK06]{ek06}
M{\'a}rton Elekes and Tam{\'a}s Keleti.
\newblock Borel sets which are null or non-{$\sigma$}-finite for every
  translation invariant measure.
\newblock {\em Adv. Math.}, 201(1):102--115, 2006.

\bibitem[Fal86]{fal86}
K.~Falconer.
\newblock {\em The geometry of fractal sets}, volume~85 of {\em Cambridge
  Tracts in Mathematics}.
\newblock Cambridge University Press, Cambridge, 1986.

\bibitem[Fal03]{fal03}
Kenneth Falconer.
\newblock {\em Fractal geometry}.
\newblock John Wiley \& Sons Inc., Hoboken, NJ, second edition, 2003.
\newblock Mathematical foundations and applications.

\bibitem[FK17]{fk17}
M.~Fujiwara and S.~Kakeya.
\newblock On some problems of maxima and minima for the curve of
  constantbreadth and the in-revolvable curve of the equilateral triangle.
\newblock {\em Tohoku Mathematical Journal}, 11:92--110, 1917.

\bibitem[Fur70]{fur70}
Harry Furstenberg.
\newblock Intersections of {C}antor sets and transversality of semigroups.
\newblock In {\em Problems in analysis ({S}ympos. {S}alomon {B}ochner,
  {P}rinceton {U}niv., {P}rinceton, {N}.{J}., 1969)}, pages 41--59. Princeton
  Univ. Press, Princeton, N.J., 1970.

\bibitem[GMS07]{gms07}
Ignacio Garcia, Ursula Molter, and Roberto Scotto.
\newblock Dimension functions of {C}antor sets.
\newblock {\em Proc. Amer. Math. Soc.}, 135(10):3151--3161 (electronic), 2007.

\bibitem[Hau18]{hau18}
Felix Hausdorff.
\newblock Dimension und \"au\ss eres {M}a\ss.
\newblock {\em Math. Ann.}, 79(1-2):157--179, 1918.

\bibitem[Kei99]{kei99}
U.~Keich.
\newblock On {$L\sp p$} bounds for {K}akeya maximal functions and the
  {M}inkowski dimension in {${\bf R}\sp 2$}.
\newblock {\em Bull. London Math. Soc.}, 31(2):213--221, 1999.

\bibitem[KN74]{kn74}
L.~Kuipers and H.~Niederreiter.
\newblock {\em Uniform distribution of sequences}.
\newblock Wiley-Interscience [John Wiley \& Sons], New York, 1974.
\newblock Pure and Applied Mathematics.

\bibitem[KT01]{kt01}
Nets Katz and Terence Tao.
\newblock Some connections between {F}alconer's distance set conjecture and
  sets of {F}urstenburg type.
\newblock {\em New York J. Math.}, 7:149--187 (electronic), 2001.

\bibitem[Mat87]{mat87}
Pertti Mattila.
\newblock Spherical averages of {F}ourier transforms of measures with finite
  energy; dimension of intersections and distance sets.
\newblock {\em Mathematika}, 34(2):207--228, 1987.

\bibitem[Mat95]{mat95}
Pertti Mattila.
\newblock {\em Geometry of sets and measures in {E}uclidean spaces}, volume~44
  of {\em Cambridge Studies in Advanced Mathematics}.
\newblock Cambridge University Press, Cambridge, 1995.
\newblock Fractals and rectifiability.

\bibitem[Mit02]{mit02}
Themis Mitsis.
\newblock Norm estimates for the {K}akeya maximal function with respect to
  general measures.
\newblock {\em Real Anal. Exchange}, 27(2):563--572, 2001/02.

\bibitem[MR10]{mr10}
Ursula Molter and Ezequiel Rela.
\newblock Improving dimension estimates for {F}ur\-sten\-berg-type sets.
\newblock {\em Adv. Math.}, 223(2):672--688, 01 2010.

\bibitem[MR12]{mr12}
Ursula Molter and Ezequiel Rela.
\newblock Furstenberg sets for a fractal set of directions.
\newblock {\em Proc. Amer. Math. Soc.}, 140(8):2753--2765, 2012.

\bibitem[MR13]{mr13}
Ursula Molter and Ezequiel Rela.
\newblock Small {F}urstenberg sets.
\newblock {\em J. Math. Anal. Appl.}, 400(2):475--486, 2013.

\bibitem[OR06]{or06}
L.~Olsen and Dave~L. Renfro.
\newblock On the exact {H}ausdorff dimension of the set of {L}iouville numbers.
  {II}.
\newblock {\em Manuscripta Math.}, 119(2):217--224, 2006.

\bibitem[Rog70]{rog70}
C.~A. Rogers.
\newblock {\em Hausdorff measures}.
\newblock Cambridge University Press, London, 1970.

\bibitem[Wol99a]{wol99a}
Thomas Wolff.
\newblock Decay of circular means of {F}ourier transforms of measures.
\newblock {\em Internat. Math. Res. Notices}, (10):547--567, 1999.

\bibitem[Wol99b]{wol99b}
Thomas Wolff.
\newblock Recent work connected with the {K}akeya problem.
\newblock In {\em Prospects in mathematics (Princeton, NJ, 1996)}, pages
  129--162. Amer. Math. Soc., Providence, RI, 1999.

\bibitem[Wol02]{wol02}
Thomas Wolff.
\newblock Addendum to: ``{D}ecay of circular means of {F}ourier transforms of
  measures'' [{I}nternat. {M}ath. {R}es. {N}otices {\bf 1999}, no. 10,
  547--567.
\newblock {\em J. Anal. Math.}, 88:35--39, 2002.
\newblock Dedicated to the memory of Tom Wolff.

\end{thebibliography}

\end{document}